\documentclass{amsart}
\usepackage{xcolor}
\usepackage{amsfonts,mathrsfs,amssymb,amsthm,mathptm}
\usepackage{amsmath,amscd}
\usepackage{amsmath}
\usepackage{mathptm,pslatex}
\usepackage{multirow}
\usepackage{color}
\usepackage[all]{xy}
\usepackage{url}
\usepackage{cite}
\usepackage{exscale}
\usepackage{relsize}
\usepackage{bbm}
\usepackage{enumerate}
\usepackage{paralist}
\usepackage{bm}
\usepackage{hyperref}
\hypersetup{colorlinks=true,citecolor=red,linkcolor=blue,CJKbookmarks=true}
\usepackage{amsrefs}
\usepackage{tikz-cd}
\usepackage{graphicx}
\usepackage{pstricks}
\usepackage{mathtools}
\xyoption{all}

\usepackage{paralist}
\newtheorem{Thm}{Theorem}[section]
\newtheorem{Def}[Thm]{Definition}
\newtheorem{Lem}[Thm]{Lemma}
\newtheorem{Cor}[Thm]{Corollary}
\newtheorem{Prop}[Thm]{Proposition}
\newtheorem{Eg}[Thm]{Example}
\newtheorem{Rem}[Thm]{Remark}

\newcommand{\Hom}{{\rm{Hom}}}

\newcommand{\Ext}{{\rm{Ext}}}
\newcommand{\End}{{\rm{End}}}
\newcommand{\modn}{{\rm{mod}}}

\newcommand{\rad}{{\rm{rad~}}}
\newcommand{\soc}{{\rm{soc}}}
\renewcommand{\top}{{\rm{top}}}

\newcommand{\proj}{{\rm{proj }}}

\newcommand{\gldim}{{\rm{gldim}}}
\newcommand{\domdim}{{\rm{domdim}}}
\newcommand{\add}{{\rm{add}}}

\newcommand{\Dyck}{{\rm{Dyck}}}
\newcommand{\thick}{{\rm{thick}}}
\newcommand{\cone}{{\rm{cone}}}
\newcommand{\os}{{\rm{os}}}

\newcommand{\projdim}{{\rm{proj.dim}}}

\numberwithin{equation}{section}
\allowdisplaybreaks[4]

\begin{document}
\title{Replicated algebras derived equivalent to higher Auslander algebras of type $\mathbb{A}$}
\author{Wei Xing}
\address{Uppsala University, Uppsala 75106, Sweden}
\curraddr{}
\email{wei.xing@math.uu.se}
\thanks{}

\subjclass[2010]{Primary ; Secondary }

\date{}
\dedicatory{}

\renewcommand{\thefootnote}{\alph{footnote}}
\setcounter{footnote}{-1} 

\begin{abstract}

We show that every higher Auslander algebra $A_{n+1}^d$ of type $\mathbb{A}$ such that $\gcd(n,d)=1$ is derived equivalent to a certain replicated algebra $B=B_0^{(n+d)}$. 
Moreover $\gldim B = nd$ and $B$ admits an $nd$-cluster tilting subcategory consisting of all direct sums of projective modules and injective modules.
We introduce a class of algebras called  $2$-subhomogeneous $m$-representation finite to characterize the homological properties of $B$ and give a method to obtain derived equivalences between fractionally Calabi-Yau algebras and $2$-subhomogeneous algebras using certain tilting complexes.
    
\end{abstract}

\maketitle

\tableofcontents


\section{Introduction}

Auslander-Reiten theory is an essential tool to study representation theory of finite dimensional algebras from the perspective of homological algebra. 
A higher version of Auslander-Reiten theory was introduced bu Iyama \cite{Iya07a, Iya07b} which is connected to algberaic geometry \cite{IW11,IW13,IW14,HIMO23}, combinatorics \cite{OT12} and symplectic geometry \cite{DJL21}. It is also a crucial ingredient to prove the Donovan-Weymss conjecture \cite{JKM22}.   
In higher Auslander-Reiten theory, the object of study is some category $\mathcal{A}$,
usually the module category of a finite dimensional algebra or its bounded derived category, 
equipped with a $d$-cluster tilting subcategory $\mathcal{M} \subseteq \mathcal{A}$,
possibly with some additional properties.
Depending on different settings, $d$-cluster tilting subcategories give rise to higher notions in homological algebra.
For instance, if $\mathcal{A}$ is abelian and $\mathcal{M}$ is $d$-cluster tilting,
then $\mathcal{M}$ is a $d$-abelian category in the sense of Jasso \cite{Jas16}.
If $\mathcal{A}$ is triangulated and $\mathcal{M}$ is $d\mathbb{Z}$-cluster tilting  which means $\mathcal{M}\subseteq \mathcal{A}$ is a $d$-cluster tilting subcategory with the additional property that $\mathcal{M}$ is closed under the $d$-fold suspension functor, then $\mathcal{M}$ is $(d+2)$-angulated in the sense of Geiss-Keller-Oppermann \cite{GKO13}.

Let $A$ be a finite dimensional algebra and denote by $\modn A$ the category of finitely generated right $A$-modules. $A$ is called $d$-representation finite if $\gldim A = d$ and there exists a $d$-cluster tilting subcategory $\mathcal{M}$ in $\modn A$. In this case, $\mathcal{M}$ is unique and canonically induces a $d\mathbb{Z}$-cluster tilting subcategory $\mathcal{U}$ in $\mathcal{D}^b(A)$. More precisely, 
\begin{equation*}
    \mathcal{U} = \add \{M[di]\mid M\in \mathcal{M} \text{ and } i\in \mathbb{Z}\}.
\end{equation*}

Such algebras are extensively studied by many authors \cite{Iya11, IO11, IO13, HI11a, HI11b}. 
It was shown in \cite{HI11a} that a $d$-representation finite algebra $A$ is twisted fractionally Calabi-Yau. 
Moreover, a finite dimensional algebra $A$ is twisted $\frac{d(\ell-1)}{\ell}$-Calabi-Yau with $\gldim A\leq d$ if and only if $A$ is $\ell$-homogeneous $d$-representation finite. Here $\ell$-homogeneous means that $\tau_d^{-(\ell-1)}P\in \add DA$ for all indecomposable projective module $P$ where $\tau_d$ denotes the $d$-Auslander-Reiten translation.
We introduce a subclass of $d$-reprensentation finite algebras, which are called $2$-subhomogeneous. They satisfy the condition that $\tau_d^{-}P\in \add DA$ for all indecomposable projective non-injective modules $P$.

\begin{Def}(Definition \ref{def_2_subhomoge_d_rf})
    Let $A$ be a finite-dimensional $k$-algebra. We call $A$ $2$-subhomogeneous $d$-representation finite if $\gldim A = d$ and $A\oplus DA$ is a $d$-cluster tilting module.
\end{Def}

Such algebras appear in the classification of $d$-representation finite acyclic Nakayama algebras. Indeed, all of them are $2$-subhomogeneous \cite{Vas19}. A similar result holds for acyclic higher Nakayama algebras \cite{Xin25}. 
In this paper, we give a general  construction of $2$-subhomogeneous $d$-representation finite algebras. 
More specifically, such an algebra is obtained as the endomorphism algebra of a certain tilting complex in the bounded derived category $\mathcal{D}^b(A)$ where $A$ is a fractionally Calabi-Yau algebra. 

\begin{Thm}(Theorem \ref{thm_construction_minimal_dRF_alg}) \label{thm_main_construction}
    Let $A$ be a finite dimensional $k$-algebra such that $\gldim A < \infty$ and $A$ is $\frac{d}{a+1}$-Calabi-Yau. 
    Suppose $X\in \mathcal{D}^b(A)$ such that $T=\bigoplus_{i=0}^{a-1}\nu^iX$ is a tilting complex.
    Denote by $B=\End_{D^b(A)}(T)$ the endomorphism algebra of $T$.
    Then the following statements hold. 
    \begin{itemize}
        \item [(i)] There is a triangle equivalence $F: \mathcal{D}^b(A)\xrightarrow[]{\sim} \mathcal{D}^b(B)$ which restricts to an equivalence between 
        $d\mathbb{Z}$-cluster tilting subcategories $F: \mathcal{U}_d(T)\xrightarrow[]{\sim}\mathcal{U}_d(B)$.
        \item [(ii)] $B$ is $2$-subhomogeneous $d$-representation finite.
    \end{itemize}
\end{Thm}

We show that our construction applies to all $(d-1)$-Auslander algebras $A_{n+1}^d$ of type $\mathbb{A}$ such that $n$ and $d$ are coprime. 
In this case, we manage to find a tilting complex which is of the form in Theorem \ref{thm_main_construction} where $X$ is projective. The endomorphism algebra $B$ is $2$-subhomogeneous $nd$-representation finite by Theorem \ref{thm_main_construction}. As a consequence, we obtain an $nd\mathbb{Z}$-cluster tilting subcatgeory of $\mathcal{D}^b(A_{n+1}^d)$ which was unknown before. 

\begin{Thm}(Theorem \ref{thm_main_construction_higher_auslander_alg_typeA}) \label{thm_main_aus_alg_type_A}
        Let $A = A_{n+1}^d$ be a $(d-1)$-Auslander algebra of type $\mathbb{A}$ with $\gcd(n, d) = 1$. There is a tilting complex $T = \bigoplus_{i=1}^{n+d}\nu^iP$ where $P$ is the basic projective module with $|P| = \frac{1}{n+d}\binom{n+d}{d}$ defined in Definition \ref{def_the_proj_module}. Denote by $B = \End_{\mathcal{D}^b(A)}(T)$ the endomorphism algebra of $T$. Then the following statements hold. 
        \begin{itemize}
            \item [(i)] There is a triangle equivalence $F: \mathcal{D}^b(A) \xrightarrow[]{\sim} \mathcal{D}^b(B)$ which restricts to an equivalence between  $nd\mathbb{Z}$-cluster tilting subcategories $F: \mathcal{U}_{nd}(T)\xrightarrow[]{\sim}\mathcal{U}_{nd}(B)$.
            \item [(ii)] $B$ is $2$-subhomogeneous $nd\mathbb{Z}$-representation finite, i.e. $\gldim B = nd$ and $B\oplus DB$ is an $nd\mathbb{Z}$-cluster tilting module.
            \item [(iii)] $B$ is $\frac{nd}{n+d+1}$-Calabi-Yau.
        \end{itemize}
    \end{Thm}

The derived equivalence of higher Auslander algebras of type $\mathbb{A}$ is of particular interest due to its connection with symplectic geometry \cite{DJL21, Ded23} and fractionally Calabi-Yau lattices \cite{Cha23, Got24}. The tilting object we constructed might be helpful in understanding the  partially wrapped Fukaya categories of the $d$-fold symmetric product of the $2$-dimensional unit disk with finitely many stops on its boundary \cite{DJL21}.

For a $d$-representation-finite algebra, its $d$-Auslander algebra \cite{Iya07b} and $(d+1)$-preprojective algebra \cite{IO13} enjoy nice homological properties.
We study such algebras associated to the endomorphism algebra $B$ obtained in Theorem \ref{thm_main_aus_alg_type_A}.
It turns out that these algebras are  closely related to the algebra $B_0 = \End_{\mathcal{D}^b(A)}(P)$. Indeed, $B$ (respectively the $nd$-Auslander algebra of $B$ ) can be described as the $(n+d)$-replicated algebra $B_0^{(n+d)}$ (respectively $(n+d+1)$-replicated algebra $B_0^{(n+d+1)}$) of $B_0$. 
The $(nd+1)$-preprojective algebra of $B$ is shown to be the $(n+d)$-fold trivial extension $T_{n+d}(B_0)$ of $B_0$. We obtain the following Theorem.  

\begin{Thm}(Corollary \ref{cor_B0_repetative_r_fold_trivial_ext}, Proposition \ref{prop_gldim_B0_fCY_B0}, Proposition \ref{prop_hig_aus_alg_of_B}) Assume $\gcd(n, d) = 1$. Let $s = \ulcorner \frac{d}{n} \urcorner$. The following statements hold true.
\begin{itemize}
    \item [(i)] $\gldim B_0 = d-s$ and $B_0$ is $\frac{(n-1)(d-1)}{n+d+1}$-Calabi-Yau.
    \item [(ii)] $B_0^{(n+d)}$ is $nd$-representation finite and $\frac{nd}{n+d+1}$-Calabi-Yau.
    \item [(iii)] $\gldim B_0^{(n+d+1)}\leq nd+1\leq \domdim B_0^{(n+d+1)}$.  
    \item [(iv)] $T_{n+d}(B_0)$ is self-injective. Moreover, $\underline{\modn} T_{n+d}(B_0)$ is $(nd+1)$-Calabi-Yau and admits an $(nd+1)$-cluster tilting module.
\end{itemize}
    
\end{Thm}

\section{Preliminaries}

\subsection{Conventions and notations}
Throughout this paper,
we fix positive integers $d$ and $n\geq 2$.
We work over an arbitrary field $k$.
For a quiver $Q$, the concatenation $pq$ of paths $p$, $q$ means firstly $q$ then $p$.
Unless stated otherwise,
all algebras are finite dimensional $k$-algebras and all modules are finite dimensional right modules.
We denote by $D$ the $k$-duality $\Hom_k(-, k)$.

Let $A$ be a finite dimensional algebra over $k$ and $\modn A$ the category of finitely generated right $A$-modules.
We denote by $\underline{\modn} A$ the projectively stable module category of $A$,
that is the category with the same  objects as $\modn A$ and morphisms given by $\underline{\Hom}_A(M,N) = \Hom_A(M,N)/\mathcal{P}(M,N)$ where $\mathcal{P}(M, N)$ denotes the subspace of morphisms factoring through projective modules.
We denote by $\Omega: \underline{\modn}A\rightarrow \underline{\modn}A$ the syzygy functor defined by $\Omega(M)$ being the kernel of the projective cover $P(M)\twoheadrightarrow M$.
Let $\Omega^0(M) = M$ and $\Omega^{i+1}(M) = \Omega(\Omega^i(M))$ for $i\geq 0$.
The injectively stable module category $\overline{\modn}A$ of $A$ and the cosyzygy functor $\Omega^{-}: \overline{\modn}A\rightarrow \overline{\modn}A$ are defined dually.

We consider the $d$-Auslander-Reiten translations $\tau_d: \underline{\modn}A\rightarrow \overline{\modn}A$    and $\tau_d^{-1}: \overline{\modn}A\rightarrow \underline{\modn}A$
defined by $\tau_d = \tau\Omega^{d-1}$ and $\tau_d^{-1} = \tau^{-1}\Omega^{-(d-1)}$
where $\tau$ and $\tau^{-1}$ denote the usual Auslander-Reiten translations.

Let $\mathcal{T}$ be a triangulated category.
We denote by $[1]$ the suspension functor of $\mathcal{T}$.
By $\thick\langle T\rangle$ we mean the smallest thick triangulated subcategory of $\mathcal{T}$ generated by $T\in \mathcal{T}$.
For a finite dimensional algebra $A$, we denote by $\proj A$ the additive category of finitely generated projective $A$-modules and $\mathcal{K}^b(\proj A)$ the homotopy category of bounded complexes of $\proj A$. 
Moreover, we denote by $\mathcal{D}^b(A)$ the bounded derived category of $A$.

\subsection{$d$-cluster tilting subcategories}

Let $\mathcal{M}$ be a subcategory of a category $\mathcal{C}$ and let $C\in \mathcal{C}$.
A right $\mathcal{M}$-approximation of $C$ is a morphism $f: M\rightarrow C$ with $M\in \mathcal{M}$ such that all morphisms $g: M'\rightarrow C$ with $M'\in \mathcal{M}$ factor through $f$. We say that 
$\mathcal{M}$ is contravariantly finite in $\mathcal{C}$ if every $C\in \mathcal{C}$ admits a right $\mathcal{M}$-approximation.
The notions of left $\mathcal{M}$-approximation and covariantly finite are defined dually.
We say that $\mathcal{M}$ is functorially finite in $\mathcal{C}$ if $\mathcal{M}$ is both contravariantly finite and covariantly finite.
In particular, if $M\in \modn A$, 
then $\add M$ is functorially finite.
Recall in the case when $\mathcal{C}$ is abelian, $\mathcal{M}$ is called a generating (resp. cogenerating) subcategory if for any object $C\in\mathcal{C}$, there exists an epimorphism $M\rightarrow C$ (resp. monomorphism $C \rightarrow M$) with $M\in \mathcal{M}$. In particular, $\add M\subset \modn A$ is called generating if and only if $A\in \add M$ and cogenerating if and only if $DA\in \add M$.

\begin{Def}[\cite{Iya11, IY08, IJ17}]
    Let $d$ be a positive integer.
    Let $\mathcal{C}$ be an abelian or a triangulated category,
    and $A$ a finite-dimensional $k$-algebra.
    \begin{itemize}
        \item [(a)] We call a subcategory $\mathcal{M}$ of $\mathcal{C}$ a $d$-cluster tilting subcategory if it is functorially finite, generating-cogenerating if $\mathcal{C}$ is abelian and 
        \begin{align*}
            \mathcal{M} & = \{ C\in \mathcal{C} \mid \Ext_{\mathcal{C}}^i(C, \mathcal{M}) = 0 \text{ for } 1\leq i\leq d-1\}\\
            & = \{ C\in \mathcal{C} \mid \Ext_{\mathcal{C}}^i(\mathcal{M}, C) = 0 \text{ for } 1\leq i\leq d-1\}.
        \end{align*}
        If moreover $\Ext_{\mathcal{C}}^i(\mathcal{M}, \mathcal{M}) \neq 0$ implies that $i\in d\mathbb{Z}$,
        then we call $\mathcal{M}$ a $d\mathbb{Z}$-cluster tilting subcategory.
        \item [(b)] A finitely generated module $M\in \modn A$ is called a $d$-cluster tilting module (respectively $d\mathbb{Z}$-cluster tilting module) if $\add M$ is a $d$-cluster tilting subcategory (respectively $d\mathbb{Z}$-cluster tilting subcategory) of $\modn A$.
    \end{itemize}
\end{Def}

Let $A$ be a finite dimensional algebra such that $\gldim A < \infty$. 
Let 
\begin{equation*}
    \nu = D\circ \mathbb{R}\Hom_{A}(-,A) \cong -\otimes^{\mathbb{L}}_A DA : \mathcal{D}^b(A) \rightarrow \mathcal{D}^b(A) 
\end{equation*}
be the Nakayama functor of $\mathcal{D}^b(A)$. There is  the functorial isomorphism \cite[Theorem 4.6]{Hap88}
\begin{equation*}
    \Hom_{\mathcal{D}^b(A)}(X,Y) \cong D\Hom_{\mathcal{D}^b(A)}(Y,\nu X),
\end{equation*}
in other words, the Nakayama functor is the Serre functor of $\mathcal{D}^b(A)$.
Denote by 
\begin{equation*}
    \nu_d = \nu\circ [-d]: \mathcal{D}^b(A)\rightarrow \mathcal{D}^b(A).
\end{equation*}
Assume that $\gldim A\leq d$, then 
\begin{equation*}
    \tau_d^{\ell}(M) \cong H^0(\nu_d^{\ell}M)
\end{equation*}
for all $\ell\in \mathbb{Z}$ and $M\in\modn A$ by \cite[Lemma 5.5]{Iya11}.
Following \cite{Iya11}, we say that $A$ is $\tau_d$-finite if moreover  $\tau_d^{\ell}(DA) = 0$ for a sufficiently large integer $\ell$.
We recall the following construction of $d$-cluster tilting subcategories in $\mathcal{D}^b(A)$. For $T\in \mathcal{D}^b(A)$, set
\begin{equation*}
    \mathcal{U}_d(T) :=\{\nu_d^i(T)\mid i\in \mathbb{Z}\} \subset \mathcal{D}^b(A).
\end{equation*}

\begin{Thm}(\cite[Theorem 1.23]{Iya11})\label{thm_Ud(T)_d-ct}
    Let $A$ be a $\tau_d$-finite algebra. Then $\mathcal{U}_d(A)$ is a $d$-cluster tilting subcategory of $\mathcal{D}^b(A)$. Moreover, $\mathcal{U}_d(T)$ is a $d$-cluster tilting subcategory of $\mathcal{D}^b(A)$ for any tilting complex $T\in \mathcal{D}^b(A)$ satisfying $\gldim \End_{\mathcal{D}^b(A)}(T)\leq d$.
\end{Thm}

We recall the definition of $d$-representation finite algebras and collect related results that are needed later.

\begin{Def}(\cite{IO11})
    Let $A$ be a finite dimensional $k$-algebra. If $\gldim A\leq d$ and $A$ has a $d$-cluster tilting module $M$. Then $A$ is called $d$-representation finite.
\end{Def}

\begin{Thm}(\cite[Theorem 3.1]{IO11}, \cite[Proposition 1.3, Theorem 1.6, Lemma 5.5(b)]{Iya11}) \label{thm_d_rf_alg}
    Let $A$ be a finite dimensional algebra with $\gldim A\leq d$. Then the followings are equivalent.
    \begin{itemize}
        \item [(i)] $A$ is $d$-representation finite;
        \item [(ii)] for each indecomposable projective module $P$, there is an integer $\ell_p\geq 1$ such that $\nu_d^{-(\ell_p-1)}P\in \add DA$.
    \end{itemize}
    In this case, we have 
    \begin{itemize}
        \item [(a)] $\tau_d^i(P) = \nu_d^i(P)$ for all $0\leq i\leq \ell_p-1$;
        \item [(b)] $\mathcal{M} = \add\{\tau_d^i(P)\mid P \text{ is indecomposable projective }, 0\leq i\leq \ell_p-1\}$ is the unique $d$-cluster tilting subcategory of $\modn A$; and 
        \item [(c)] $\mathcal{U} = \add\{M[di]\mid M\in \mathcal{M}, i\in \mathbb{Z}\} = \add\{\nu_d^i(A)\mid i\in \mathbb{Z}\} = \mathcal{U}_d(A)$ is a $d\mathbb{Z}$-cluster tilting subcategory of $\mathcal{D}^b(A)$. 
    \end{itemize}
\end{Thm}

\subsection{Higher Auslander algebras and higher preprojective algebras}
We recall certain constructions of algebras associated to an algebra $A$ with $\gldim A \leq d$, namely the $(d+1)$-preprojective algebra and the $d$-Auslander algebra in the case when $A$ is $d$-representation finite.
We collect some properties of these algebras.

\begin{Def}\cite[Definition 2.11]{IO13}
    Let $A$ be an algebra with $\gldim A\leq d$. Then the $(d+1)$-preprojective algebra of $A$ is 
    \begin{equation*}
        \Pi_{d+1}(A) = T_A\Ext_A^d(DA, A),
    \end{equation*}
    that is the tensor algebra of the $A$-$A$-bimodule $\Ext_A^d(DA, A)$.
\end{Def}

Alternatively by \cite[Lemma 2.13]{IO13} and \cite[Lemma 5.5]{Iya11}, we have
\begin{equation*}
    \Pi_{d+1}(A) \cong \bigoplus_{i\geq 0}\tau_d^{-i}A \cong \bigoplus_{i\geq 0}H^0(\nu_d^{-i}A) \cong \bigoplus_{i\geq 0}\Hom_{\mathcal{D}^b(A)}(A, \nu_d^{-i}A).
\end{equation*}

In the case when $A$ is $d$-representation finite, $\Pi_{d+1}(A)$ has nice properties as we collect with the following proposition. Recall that for a triangulated category $\mathcal{T}$ with Serre functor $\mathbb{S}$, we say $\mathcal{T}$ is $n$-Calabi-Yau if 
\begin{equation*}
    \mathbb{S} \cong [n],
\end{equation*}
as triangulated functors.

\begin{Prop}\cite[Corollary 3.4, Corollary 4.16]{IO13}\label{prop_properties_hig_preproj_alg}
    Let $A$ be a $d$-representation finite algebra and $\Pi_{d+1}(A)$ be its $(d+1)$-preprojective algebra. The following statements hold.
    \begin{itemize}
        \item [(i)] $\Pi_{d+1}(A)$ is self-injective. 
        \item [(ii)] $\underline{\modn} \Pi_{d+1}(A)$ is $(d+1)$-Calabi-Yau triangulated with a $(d+1)$-cluster tilting object.
        \item [(iii)] $\modn \Pi_{d+1}(A)$ admits a $(d+1)$-cluster tilting object  $\Hom_A(\Pi_{d+1}(A), \Pi_{d+1}(A))$.
    \end{itemize}
\end{Prop}

Now we recall the definition and properties of the $d$-Auslander algebra of a $d$-representation finite algebra.

\begin{Def}(\cite{Iya11})
    Let $A$ be a $d$-representation finite algebra with the basic $d$-cluster tilting module $M = \bigoplus_{i\geq 0}\tau_d^{-i}(A)$. 
    We call $\Lambda = \End_A(M)$ the $d$-Auslander algebra of $A$. 
\end{Def}

\begin{Prop}(\cite[Theorem 0.2]{Iya07b})\label{prop_aus_alg_gldim_domdim}
    We have $\gldim \Lambda\leq d+1 \leq \domdim \Lambda$.
\end{Prop}

\subsection{Tilting complexes and derived equivalences}

Recall that two algebras $A, B$ are derived equivalent if $\mathcal{D}^b(A) \cong \mathcal{D}^b(B)$ as triangulated categories. It follows from Rickard's Morita theory for derived categories 
\cite{Ric89} that two algebras $A$ and $B$ are derived equivalent if and only if there exists a tilting complex $T\in \mathcal{K}^b(\proj A)$ such that $B\cong \End_{\mathcal{D}^b(A)}(T)$. In this case, there is a triangle equivalence $F:\mathcal{D}^b(A)\xrightarrow{\sim} \mathcal{D}^b(B)$ such that $F(T)\cong B$.
We recall the definition of a tilting complex here. 

\begin{Def} (\cite{Ric89})
    Denote by $\mathcal{K}^b(\proj A)$ the homotopy category of  bounded complexes of finitely generated projective modules. A complex $T$ is a tilting complex if 

    \begin{itemize}
        \item [(i)] $T\in \mathcal{K}^b(\proj A)$,
        \item [(ii)] $\Hom(T, T[i]) = 0$ for $i\neq 0$, and
        \item [(iii)] $T$ generates $\mathcal{K}^b(\proj A)$, that is, $\mathcal{K}^b(\proj A)$ is the smallest thick triangulated subcategory of $\mathcal{D}^b(A)$ containing $T$.
    \end{itemize}
\end{Def}
  
\section{$2$-subhomogeneous $d$-representation finite algebras}

\subsection{Definitions and properties}
In this section, we introduce the $2$-subhomogeneous $d$-representation finite algebras. Moreover, we give some examples and a criterion for such algebras.

\begin{Def}\label{def_2_subhomoge_d_rf}
    Let $A$ be a finite dimensional $k$-algebra. We say $A$ is $2$-subhomogeneous $d$-representation finite if $\gldim A = d$ and $A\oplus DA$ is a $d$-cluster tilting module. 
\end{Def}

\begin{Rem} \label{rem_2subhomo_ell_p}
We follow the notations in Theorem \ref{thm_d_rf_alg}. Let $A$ be a finite dimensional algebra with $\gldim A\leq d$. 
    \begin{itemize}
        \item [(i)] The algebra $A$ is $2$-subhomogeneous if and only if $\ell_p\leq 2$ for all $P$.
        \item [(ii)] In \cite{HI11a}, $A$ is called $2$-homogeneous if $\ell_p = 2$ for all $P$. 
        Thus a $2$-homogeneous algebra is $2$-subhomogeneous with the condition that $\add A\cap \add DA = \{0\}$.
    \end{itemize}
\end{Rem}

\begin{Eg}
    \begin{itemize}
        \item [(i)] Such algebras appear in the classification of $d$-representation finite acyclic Nakayama algebras \cite[Theorem 3]{Vas19}. Indeed, they are all $2$-subhomogeneous $d$-representation finite.
        \item [(ii)] A similar result for acyclic higher Nakayama algebras will appear in \cite{Xin25}. More precisely, among all homogeneous acyclic $d$-Nakayama algebras, the $nd$-representation finite ones for some integer $n > 1$ are all $2$-subhomogeneous.
        \item [(iii)] Let $A_i$ be $2$-homogeneous $d_i$-representation finite for $i\in \{1,2\}$ and $k$ be a perfect field. 
        Then $A_1\otimes_k A_2$ is $2$-homogeneous $(d_1+d_2)$-representation finite and thus $2$-subhomogeneous. See \cite{HI11a} and \cite[Section 11]{Soed24} for explicit examples.
    \end{itemize}
\end{Eg}

Now we give an equivalent condition to characterize $2$-subhomogeneous $d$-representation finite algebras. 

\begin{Prop}\label{Prop_condition_2subhomo_d_RF}
   Let $A$ be a finite dimensional $k$-algebra with $\gldim A\leq d$. The following are equivalent.
   \begin{itemize}
       \item [(i)] $A$ is $2$-subhomogeneous $d$-representation finite;
       \item [(ii)] $\nu_d^{-1}(P)\in \add DA$ for all $P$ indecomposable projective non-injective; and 
       \item [(iii)] $\nu_d(I)\in \add A$ for all $I$ indecomposable injective non-projective.
   \end{itemize}
\end{Prop}
\begin{proof}
    Note that $\nu_d$ is an autoequivalence and 
    \begin{equation*}
        \nu_d^{\pm 1}(\add A\cap \add DA) \subset \modn A[\mp d],
    \end{equation*}
    thus $(ii)$ and $(iii)$ are equivalent.
    By Theorem \ref{thm_d_rf_alg} and Remark \ref{rem_2subhomo_ell_p}, we have $(i)$ and $(ii)$ are equivalent.
\end{proof}

\subsection{Construction} \label{section_construction}

We firstly recall the definition of (twisted) fractionally Calabi-Yau algebras. Then we give a construction of $2$-subhomogeneous $d$-representation finite algebras which arise as endomorphism algebras of tilting complexes over a $\frac{d}{a+1}$-Calabi-Yau algebra for some integer $a>0$. We show with an example that not all such algebras can be constructed in this way.

Let $\phi: A\rightarrow A$ be an algebra automorphism. It induces an autoequivalence 
\begin{equation*}
    \phi^{\ast} = -\otimes^{\mathbb{L}}_A {_{\phi}A}: \mathcal{D}^b(A) \rightarrow \mathcal{D}^b(A)
\end{equation*}
where $_{\phi}A$ is the $A\otimes_k A^{op}$-module $A$ with the left action changed to $a\cdot b := \phi(a)b$.

\begin{Def}\cite[Definition 0.3]{HI11a}
    We say that $A$ is twisted fractionally Calabi-Yau (or twisted $\frac{m}{\ell}$-CY) if there exists an isomorphism 
    \begin{equation*}
        \nu^{\ell} \cong [m]\circ \phi^{\ast}
    \end{equation*}
    of functors for some integers $\ell\neq 0$ and $m$ and $\phi$ an algebra automorphism of $A$. 
    When $\phi =id$, we say that $A$ is fractionally Calabi-Yau (or $\frac{m}{\ell}$-CY).
\end{Def}

\begin{Thm}\label{thm_construction_minimal_dRF_alg}
    Let $A$ be a finite dimensional algebra such that $\gldim A < \infty$ and $A$ is $\frac{d}{a+1}$-CY. 
    Suppose $X\in D^b(A)$ is such that $T=\bigoplus_{i=0}^{a-1}\nu^iX$ is a tilting complex. Denote by $B=\End_{D^b(A)}(T)$ the endomorphism algebra of $T$. Then the following statements hold. 
    \begin{itemize}
        \item [(i)] There is a triangle equivalence $F: \mathcal{D}^b(A)\xrightarrow[]{\sim} \mathcal{D}^b(B)$ which restricts to an equivalence between $d\mathbb{Z}$-cluster tilting subcategories $F: \mathcal{U}_d(T)\xrightarrow{\sim}\mathcal{U}_d(B)$.
        \item [(ii)] $B$ is $2$-subhomogeneous $d$-representation finite.
    \end{itemize}
\end{Thm}
\begin{proof}
    Since $T$ is a tilting complex, there is a triangle equivalence $F: \mathcal{D}^b(A)\rightarrow \mathcal{D}^b(B)$ such that $F(T)\cong B$. 
    We have the following commutative diagram.
    \begin{equation*}
        \xymatrix{
        D^b(A)\ar[r]^F\ar[d]^{\nu_A} & D^b(B)\ar[d]^{\nu_B}\\
        D^b(A)\ar[r]^F & D^b(B).
        }
    \end{equation*}
    Thus $DB\cong \nu_B(B) \cong \nu_B(FT) \cong F\nu_A(T)$. 

    Since finiteness of global dimension is a derived invariant, we have that $\gldim B < \infty$. Now we prove that $\gldim B = d$.
    It suffices to show $\projdim DB = d$.
    
     Let $T'= \bigoplus_{i=1}^{a-1}\nu_A^iX$. We have $T = T'\oplus X$.
    Since $A$ is $\frac{d}{a+1}$-Calabi-Yau, we have 
    \begin{equation*}
        \nu_{A,d}^{-1} =  \nu_A^{-1}[d] \cong \nu_A^{-1} \nu_A^{a+1} = \nu_A^a.
    \end{equation*}
    Thus $\nu_A T=\bigoplus_{i=1}^a\nu_A^iX = T'\oplus \nu_{A,d}^{-1}X$. 
    So we have
    
    \begin{align*}
        \Hom_{D^b(B)}(DB, B[i]) & \cong \Hom_{D^b(B)}(F\nu_A T, FT[i]) \\
        & \cong \Hom_{D^b(A)}(\nu_A T, T[i])\\
        & = \Hom_{D^b(A)}(T'\oplus \nu_{A,d}^{-1}X, T[i])\\
        & = \Hom_{D^b(A)}(T', T[i])\oplus \Hom_{D^b(A)}(\nu_{A,d}^{-1}X, T[i])\\
        & \cong \Hom_{D^b(A)}(T', T[i])\oplus \Hom_{D^b(A)}(X, \nu_AT[i-d])\\
        & \cong \Hom_{D^b(A)}(T', T[i])\oplus D\Hom_{D^b(A)}(T, X[d-i])\\
        & =\left\{\begin{array}{cc}
           \Hom_{D^b(A)}(T', T)  & i=0 \\
           D\Hom_{D^b(A)}(T, X)  & i=d\\
           0 & i\neq 0, d.
        \end{array}\right.
    \end{align*}
    This implies $\projdim DB = d$ and therefore $\gldim B = d$.

    By Theorem \ref{thm_Ud(T)_d-ct},
    \begin{equation*}
        \mathcal{U}_d(T) = \add\{\nu_{A,d}^i(T)\mid i\in \mathbb{Z}\} = \add\{\nu_A^i(X)\mid i\in \mathbb{Z}\} \subset \mathcal{D}^b(A)
    \end{equation*}
    is $d$-cluster tilting. 
    Since $\nu_A^{a+1}\cong [d]$ as $A$ is $\frac{d}{a+1}$-Calabi-Yau, we have $\mathcal{U}_d(T)[d]\subset \mathcal{U}_d(T)$ which implies $\mathcal{U}_d(T)$ is $d\mathbb{Z}$-cluster tilting.
    Moreover $F: \mathcal{U}_d(T)\xrightarrow{\sim} \mathcal{U}_d(B)$ since $F(T)\cong B$. Statement (i) follows.
    
    Since $T'\in \add \nu_AT$, we get $FT'\in \add DB$, which implies that all indecomposable projective non-injective $B$-modules lie in $\add FX$. Now $\nu_{A,d}^{-1}X\cong \nu_A^aX\in \add \nu_A T$ implies that 
    \begin{equation*}
        \nu_{B,d}^{-1}FX \cong F(\nu_A^a X) \in \add DB.
    \end{equation*}
    Therefore by Proposition \ref{Prop_condition_2subhomo_d_RF}, $B$ is $2$-subhomogeneous $d$-representation finite. 
\end{proof}

\begin{Rem}
    Indeed, $\mathcal{U}_d(T)$ has a $(d+2)$-angulated structure by \cite[Theorem 1]{GKO13} essentially induced by the triangulated structure of $\mathcal{D}^b(A)$. Moreover the restriction 
    \begin{equation*}
        F: \mathcal{U}_d(T) \xrightarrow{\sim} \mathcal{U}_d(B)
    \end{equation*}
    is an equivalence of $(d+2)$-angulated categories.
\end{Rem}

\begin{Rem} Fixing notations as above, we make some further observations.
     The same construction applies if $A$ is twisted $\frac{d}{a+1}$-CY and $\phi^{\ast}(X)\cong X$. In this case, 
     \begin{equation*}
         \nu_{A,d}^{-1} \circ \phi^{\ast} = \nu_A^{-1}[d]\circ \phi^{\ast} \cong \nu_A^{-1}\circ \nu_A^{a+1}  = \nu_A^a,
     \end{equation*}
     and
     \begin{equation*}
         \nu_AT = \bigoplus_{i=1}^a\nu_A^iX = T'\oplus \nu_{A,d}^{-1}\circ \phi^{\ast}(X) \cong T'\oplus \nu_{A,d}^{-1}(X).
     \end{equation*}
\end{Rem}

Next we describe the endomorphism algebra $B$. 
For simplicity, we assume $X$ is basic. 
Write $B_i = \Hom_{\mathcal{D}^b(A)}(X, \nu_A^iX)$. 
By Serre duality, $\Hom_{\mathcal{D}^b(A)}(\nu_A^iX, X) \cong DB_{i+1}$ and $B_1\cong DB_0$. 
Thus
\begin{equation*}
    B\cong 
    \begin{pmatrix}
        B_0 & B_1 & \cdots & B_{a-1}\\
        DB_2 & B_0 & \cdots & B_{a-2}\\
        \vdots & \vdots & \ddots & \vdots\\
        DB_a & DB_{a-1} & \cdots & B_0
    \end{pmatrix}
    \cong
    \begin{pmatrix}
        B_0 & DB_0 & \cdots & B_{a-1}\\
        DB_2 & B_0 & \cdots & B_{a-2}\\
        \vdots & \vdots & \ddots & \vdots\\
        DB_a & DB_{a-1} & \cdots & B_0
    \end{pmatrix}.
\end{equation*}

Since $\add T\cap \add \nu_A T = \add \bigoplus_{i=1}^{a-1}\nu_A^iX$, we have that $F\nu_A^{a}X$ is the direct sum of all indecomposable injective non-projective $B$-modules. 
As $B$ is $2$-subhomogeneous $d$-representation finite, 
\begin{equation*}
    N = B\oplus F\nu_A^aX \cong F\left(\bigoplus_{i=0}^a\nu_A^iX\right)
\end{equation*}
is the basic $d$-cluster tilting module. 
Denote by $\Lambda=\End_B(N)$ the $d$-Auslander algebra of $B$.
Then 
\begin{equation*}
    \Lambda \cong \End_{\mathcal{D}^b(A)}\left(\bigoplus_{i=0}^a \nu_A^iX\right).
\end{equation*}
A similar argument as for $B$ gives 
\begin{equation*}
    \Lambda \cong 
    \begin{pmatrix}
        B_0 & DB_0 & \cdots & B_{a}\\
        DB_2 & B_0 & \cdots & B_{a-1}\\
        \vdots & \vdots & \ddots & \vdots\\
        DB_{a+1} & DB_{a} & \cdots & B_{0}
    \end{pmatrix}.
\end{equation*}

\begin{Eg}\label{example_linear_A4}
    Let $A = kQ^{4,1}$ be the path algebra of the quiver 
    $Q^{4,1}: \xymatrix@C=1em@R=1em{1\ar[r] & 2\ar[r] & 3\ar[r] & 4}$. We have $\gldim A = 1$ and $A$ is $\frac{3}{5}$-Calabi-Yau. Denote by $P_i$ the indecomposable projective module at vertex $i$. Let 
    \begin{equation*}
        T = \bigoplus_{i=0}^3 \nu^iP_1 = P_1\oplus P_4\oplus (\xymatrix@R=.8em@C=.8em{P_3\ar@{^{(}->}[r] & P_4})\oplus (\xymatrix@R=.8em@C=.8em{P_2\ar@{^{(}->}[r] & P_3})[1].
    \end{equation*}
    It can be checked directly that $\Hom_{\mathcal{D}^b(A)}(T, T[i]) = 0$ for $i\neq 0$.
    Moreover 
    \begin{align*}
        P_3 & = \cone (P_4 \rightarrow (\xymatrix@R=.8em@C=.8em{P_3\ar@{^{(}->}[r] & P_4}))[-1]\\
        P_2 & = \cone (P_3 \rightarrow (\xymatrix@R=.8em@C=.8em{P_2\ar@{^{(}->}[r] & P_3}))[-1]
    \end{align*}
    which implies $\thick\langle T\rangle = \mathcal{K}^b(\proj A)$. 
    Therefore $T$ is a tilting complex.
    So by Theorem \ref{thm_construction_minimal_dRF_alg}, we conclude that the endomorphism algebra $B = \End_{\mathcal{D}^b(A)}(T)$ is $2$-subhomogeneous $3$-representation finite. Indeed $B\cong A/\rad^2(A)$ is the Koszul dual of $A$.
\end{Eg}

\begin{Eg}\label{example_A_4_2}
    Let $\Lambda=kQ^{4,2}/I_{4,2}$ be the Auslander algebra of $A$ in Example \ref{example_linear_A4}. 
    More precisely, $Q^{4,2}$ is given as follows. 
     The admissible ideal $I_{4,2}$ is generated by commutative relations of squares (blue dashed line) and zero relations of half squares (red dashed line).

    
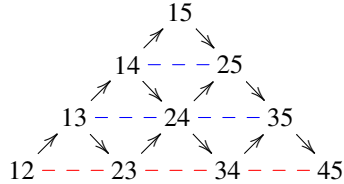
\begin{equation*}
        \begin{xy}
            0;<26pt,0cm>:<20pt,20pt>::
            (0,0) *+{12} ="12",
            (0,1) *+{13} ="13",
            (0,2) *+{14} ="14",
            (0,3) *+{15} ="15",
            (1.5,0) *+{23} ="23",
            (1.5,1) *+{24} ="24",
            (1.5,2) *+{25} ="25",
            (3,0) *+{34} ="34",
            (3,1) *+{35} ="35",
            (4.5,0) *+{45} ="45",
            "12", {\ar "13"},
            "13", {\ar "14"},
            "14", {\ar "15"},
            "13", {\ar "23"},
            "14", {\ar "24"},
            "15", {\ar "25"},
            "23", {\ar "24"},
            "24", {\ar "25"},
            "24", {\ar "34"},
            "25", {\ar "35"},
            "34", {\ar "35"},
            "35", {\ar "45"},
            "13", {\ar@[blue]@{--} "24"},
            "14", {\ar@[blue]@{--} "25"},
            "24", {\ar@[blue]@{--} "35"},
            "12", {\ar@[red]@{--} "23"},
            "23", {\ar@[red]@{--} "34"},
            "34", {\ar@[red]@{--} "45"},
        \end{xy}
    \end{equation*}

    We have that $\gldim \Lambda = 2$ and $\Lambda$ is $\frac{6}{6}$-Calabi-Yau, see \cite[Remark 2.29]{DJW19} \cite[Theorem 6.21]{Gra23}. 
    Let $T =\bigoplus_{i=0}^4 \nu^i (P_{12}\oplus P_{13})$. It can be checked directly that $T$ is a tilting complex and the endomorphism algebra $\Gamma = \End_{\mathcal{D}^b(\Lambda)}(T) \cong kQ^{10,1}/\rad^3(kQ^{10,1})$ where 
    \begin{equation*}
        Q^{10,1}: \xymatrix@C=.8em@R=.8em{1\ar[r] & 2\ar[r] & 3\ar[r] & \cdots\ar[r] & 9\ar[r] & 10}.
    \end{equation*}
    By Theorem \ref{thm_construction_minimal_dRF_alg}, $\Gamma$ is $2$-subhomogeneous $6$-representation finite. Note that $\Gamma$ also appears in \cite[Theorem 3]{Vas19}.
\end{Eg}

\begin{Rem}
    \begin{itemize}
        \item [(i)] The tilting complex in Example \ref{example_linear_A4} appears implicitly in the classification of $n\mathbb{Z}$-cluster tilting subcategories for a self-injective Nakayama algebra \cite{HKV25}.
        \item [(ii)] The derived equivalence between $\Lambda$ and $\Gamma$ in Example \ref{example_A_4_2} was obtained in \cite[Corollary 1.2, Corollary 1.13]{Lad12}. Here we give an explicit description of a tilting complex which induces such a derived equivalence.
    \end{itemize}
\end{Rem}

We give an example here to show that not all $2$-subhomogeneous $d$-representation finite algebras can be obtained via our construction.
\begin{Eg} 
Let $A = kQ/I$ where $Q$ is given below and $I$ is generated by commutative relations of squares (blue dashed lines) and zero relations of half squares (red dashed lines).
It is shown in \cite{Xin25} that $\gldim A = 6$ and $A\oplus DA$ is a $6$-cluster tilting module. In other words, $A$ is $2$-subhomogeneous $6$-representation finite. 
But it is not possible to write $A=\bigoplus_{i=0}^{N-1}\nu^i P$ where $P$ is the direct sum of projective non-injective modules. Since $P = P_1\oplus P_2\oplus P_4$, the only option is $N=4$. But one can check that $\nu^3 P_4$ is not projective.
    
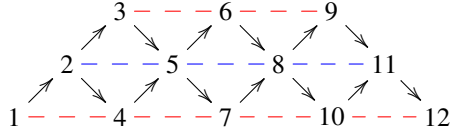
\begin{equation*}
        \begin{xy}
            0;<20pt,0cm>:<20pt,20pt>::
            (0,0) *+{1} ="12",
            (0,1) *+{2} ="13",
            (0,2) *+{3} ="14",
            (2,0) *+{4} ="23",
            (2,1) *+{5} ="24",
            (2,2) *+{6} ="25",
            (4,0) *+{7} ="34",
            (4,1) *+{8} ="35",
            (4,2) *+{9} ="36",
            (6,0) *+{10} ="45",
            (6,1) *+{11} ="46",
            (8,0) *+{12} ="56",
            "12", {\ar "13"},
            "13", {\ar "14"},
            "23", {\ar "24"},
            "24", {\ar "25"},
            "34", {\ar "35"},
            "35", {\ar "36"},
            "45", {\ar "46"},
             "13", {\ar "23"},
             "14", {\ar "24"},
             "24", {\ar "34"},
             "25", {\ar "35"},
             "35", {\ar "45"},
             "36", {\ar "46"},
             "46", {\ar "56"},
             "12", {\ar@[red]@{--} "23"},
             "14", {\ar@[red]@{--} "25"},
             "23", {\ar@[red]@{--} "34"},
             "25", {\ar@[red]@{--} "36"},
             "34", {\ar@[red]@{--} "45"},
             "45", {\ar@[red]@{--} "56"},
             "13", {\ar@[blue]@{--} "24"},
             "24", {\ar@[blue]@{--} "35"},
             "35", {\ar@[blue]@{--} "46"},
        \end{xy}
    \end{equation*}
\end{Eg}

\section{A derived equivalence for higher Auslander algebras of type A}

\subsection{Higher Auslander algebras of type A}
In this section, we give a description of higher Auslander algebras of type $\mathbb{A}$ by quivers with relations. We state the main theorem and illustrate it with an example.

We recall the definition of ordered sequences $(\os_n^d, \preccurlyeq)$ from \cite{JK19}.
\begin{equation*}
        \os_n^d := \{x = (x_1, x_2, \ldots, x_d)  \mid 1\leq x_1 < x_2 < \cdots < x_d \leq n+d-1\},
\end{equation*}
with the relation $\preccurlyeq$ defined as $x \preccurlyeq y$ if $x_1 \leq y_1 < x_2\leq y_2 < \cdots < x_d \leq y_d$ for $x=(x_1, \ldots, x_d)$, $y = (y_1, \ldots, y_d) \in \os_n^d$. 

Now we describe the $d$-Auslander algebras of type $\mathbb{A}$ by quiver with relations.
The vertex set of the quiver $Q^{n,d}$ is given by $\os_n^d$.
Let $\{e_i\mid 1\leq i\leq d\}$ be the standard basis of $\mathbb{Z}^d$.
There is an arrow $a_i(x) : x\rightarrow x + e_i$ whenever $x + e_i \in \os_n^d$.
Let $I_{n,d}$ be the ideal of the path category $kQ^{n,d}$ generated by  $a_i(x + e_j)a_j(x) - a_j(x + e_i)a_i(x)$ with $1\leq i,j\leq d$. 
By convention, $a_i(x) = 0$ whenever $x$ or $x+e_i$ is not in  $\os_n^d$, hence some of the relations are indeed zero relations.
Then the $(d-1)$-Auslander algebra of type $\mathbb{A}_n$ is given by $A_n^d = kQ^{n,d} / I_{n,d}$.

\begin{Rem}
        By definition $A_n^d$ is a locally bounded $k$-linear category.
        By abuse of notation,
        we still call it an algebra.
        We also identify $k$-linear categories with finitely many objects and algebras.
\end{Rem}
By construction in \cite{Iya11}, $A_n^d$ has global dimension $d$ and a distinguished $d$-cluster tilting subcategory  $$\mathcal{M}_n^d = \add \{M(x) \mid x\in \os_n^{d+1}\}.$$
Here as a representation $M(x)$ assigns to vertex $z\in \os_n^{d}$ the vector space
\begin{equation*}
    M(x)_z \cong \left\{ \begin{array}{ll}
        k &  \text{ if } (x_1, \ldots, x_d) \preccurlyeq z \preccurlyeq (x_2-1, \ldots, x_{d+1}-1)\\
        0 & \text{ otherwise.}
    \end{array} 
    \right.
\end{equation*}

Then
$$\Hom_{A_n^d}(M(x),M(y)) \cong \left\{ \begin{array}{ll}
        kf_{yx} &  x \preccurlyeq y\\
        0 & \text{ otherwise.}
    \end{array} 
    \right.$$
Here $f_{yx}$ is given by $k\xrightarrow{1} k$ at vertices $z$ where $M(x)_z = M(y)_z = k$ and $0$ otherwise.
The composition of morphisms in $\mathcal{M}_n^d$ is completely determined by
\begin{equation*}
    f_{zy}\circ f_{yx} =
    \left\{ \begin{array}{ll}
        f_{zx} &  x \preccurlyeq z\\
        0 & \text{ otherwise.}
    \end{array} 
    \right.
\end{equation*}

In fact, $A_n^{d+1}$ is the $d$-Auslander algebra of $A_n^d$.

We collect some properties of $\mathcal{M}_n^d$ in the following property.
\begin{Prop}(\cite[Theorem 3.6, Proposition 3.17, Proposition 3.19]{OT12}, \cite[Remark 2.29]{DJW19}, \cite[Theorem 6.2]{Gra23}) Let $A=A_n^d$. The following properties hold. \label{prop_A_n_d}
    \begin{itemize}
        \item [(i)] $M(x_1, \ldots, x_{d+1})$ is projective if  and only if $x_1=1$ and injective if and only if  $x_{d+1} = n+d$. 
        \item [(ii)] We have $\soc M(x_1, \ldots, x_{d+1}) = S_{(x_1, \ldots, x_d)}$ and $\top M(x_1, \ldots, x_{d+1}) = S_{(x_2-1, \ldots, x_{d+1}-1)}$ where $S_y$ denotes the simple module at vertex $y\in \os_{n}^d$.
        \item [(ii)] If $M(x)$ is not projective then $\tau_dM(x) = M(\tau_d x)$ where $\tau_d (x_1, \ldots, x_{d+1}) = (x_1-1, \ldots, x_{d+1}-1).$
        \item [(iii)] 
            \begin{equation*}
                \Ext_{A_n^d}^i(M(x),M(y)) \cong \left\{ \begin{array}{ll}
                kg_{yx} &  y \preccurlyeq \tau_d(x) \text{ and } i=d\\
                0 & \text{ otherwise.}
                \end{array} \right.
            \end{equation*}
        \item [(iv)] Given $1\leq x_1<x_2\cdots<x_{d+1}<x_{d+2}\leq n+d$, we have the following $(d+2)$-exact sequence, which will be called a  minimal $(d+2)$-exact sequence.
        \begin{equation*}
            \xymatrix@C=1em@R=1em{
            0\ar[r] & M(x_1, \ldots, x_{d+1})\ar[r] & \cdots\ar[r] & M(x_1,\ldots,x_{i-1},x_{i+1},\ldots,x_{d+2})\ar[r] & M(x_2,\ldots,x_{d+2})\ar[r] & 0.  
            }
        \end{equation*}
        In particular, when $x_1=1$, such a sequence gives a minimal projective resolution of $M(x_2, \ldots, x_{d+2})$.
        \item [(iv)] $A$ is $\frac{(n-1)d}{n+d}$-CY.
    \end{itemize} 
\end{Prop}

    \begin{Eg}
     The quiver $Q^{5,3}$ of $A=A^{3}_{5}$ is given as follows. The ideal $I_{5,3}$ is generated by the commutativity relations of squares (blue dashed line) and the zero relations of half squares (red dashed line). Most labels and relations are omitted for simplicity.
     
    \begin{equation*}
    {\tiny
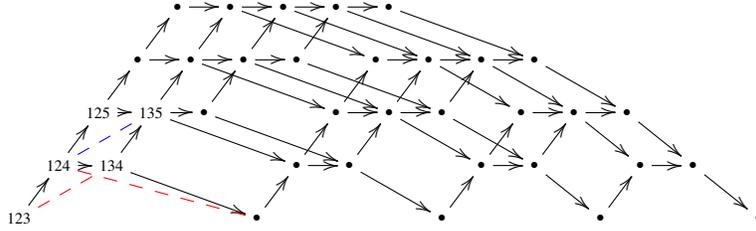
    \begin{xy}
    0;<20pt,0cm>:<15pt,20pt>:: 
    (0,0) *+{123} ="123",
    (0,1) *+{124} ="124",
    (0,2) *+{125} ="125",
    (0,3) *+{\bullet} ="126",
    (0,4) *+{\bullet} ="127",
    (1,1) *+{134} ="134",
    (1,2) *+{135} ="135",
    (1,3) *+{\bullet} ="136",
    (1,4) *+{\bullet} ="137",
    (2,2) *+{\bullet} ="145",
    (2,3) *+{\bullet} ="146",
    (2,4) *+{\bullet} ="147",
    (3,3) *+{\bullet} ="156",
    (3,4) *+{\bullet} ="157",
    (4,4) *+{\bullet} ="167",
    (4.5,0) *+{\bullet} ="234",
    (4.5,1) *+{\bullet} ="235",
    (4.5,2) *+{\bullet} ="236",
    (4.5,3) *+{\bullet} ="237",
    (5.5,1) *+{\bullet} ="245",
    (5.5,2) *+{\bullet} ="246",
    (5.5,3) *+{\bullet} ="247",
    (6.5,2) *+{\bullet} ="256",
    (6.5,3) *+{\bullet} ="257",
    (7.5,3) *+{\bullet} ="267",
    (8,0) *+{\bullet} ="345",
    (8,1) *+{\bullet} ="346",
    (8,2) *+{\bullet} ="347",
    (9,1) *+{\bullet} ="356",
    (9,2) *+{\bullet} ="357",
    (10,2) *+{\bullet} ="367",
    (11,0) *+{\bullet} ="456",
    (11,1) *+{\bullet} ="457",
    (12,1) *+{\bullet} ="467",
    (14,0) *+{\bullet} ="567",
    "123", {\ar"124"},
    "124", {\ar"125"},
    "125", {\ar"126"},
    "126", {\ar"127"},
    "124", {\ar"134"},
    "125", {\ar"135"},
    "126", {\ar"136"},
    "127", {\ar"137"},
    "134", {\ar"135"},
    "135", {\ar"136"},
    "136", {\ar"137"},
    "135", {\ar"145"},
    "136", {\ar"146"},
    "137", {\ar"147"},
    "134", {\ar"234"},
    "135", {\ar"235"},
    "136", {\ar"236"},
    "137", {\ar"237"},
    "145", {\ar"146"},
    "146", {\ar"147"},
    "146", {\ar"156"},
    "147", {\ar"157"},
    "145", {\ar"245"},
    "146", {\ar"246"},
    "147", {\ar"247"},
    "156", {\ar"157"},
    "157", {\ar"167"},
    "156", {\ar"256"},
    "157", {\ar"257"},
    "167", {\ar"267"},
    "234", {\ar"235"},
    "235", {\ar"236"},
    "236", {\ar"237"},
    "235", {\ar"245"},
    "236", {\ar"246"},
    "237", {\ar"247"},
    "245", {\ar"246"},
    "246", {\ar"247"},
    "246", {\ar"256"},
    "247", {\ar"257"},
    "245", {\ar"345"},
    "246", {\ar"346"},
    "247", {\ar"347"},
    "256", {\ar"257"},
    "257", {\ar"267"},
    "256", {\ar"356"},
    "257", {\ar"357"},
    "267", {\ar"367"},
    "345", {\ar"346"},
    "346", {\ar"347"},
    "356", {\ar"357"},
    "346", {\ar"356"},
    "347", {\ar"357"},
    "357", {\ar"367"},
    "356", {\ar"456"},
    "357", {\ar"457"},
    "367", {\ar"467"},
    "456", {\ar"457"},
    "457", {\ar"467"},
    "467", {\ar"567"},
    "123", {\ar@[red]@{--}"134"},
    "124", {\ar@[red]@{--}"234"},
    "124", {\ar@[blue]@{--}"135"},
    \end{xy}}
    \end{equation*}
    \end{Eg}

    \begin{Eg} Let $x=(1247)\in \os_5^4$. We have the module $M(x)$ as follows. 
        \begin{equation*}
    {\tiny
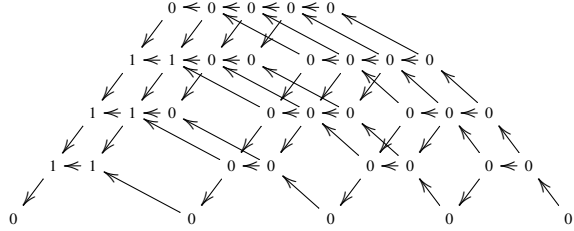
    \begin{xy}
    0;<15pt,0cm>:<15pt,20pt>:: 
    (0,0) *+{0} ="123",
    (0,1) *+{1} ="124",
    (0,2) *+{1} ="125",
    (0,3) *+{1} ="126",
    (0,4) *+{0} ="127",
    (1,1) *+{1} ="134",
    (1,2) *+{1} ="135",
    (1,3) *+{1} ="136",
    (1,4) *+{0} ="137",
    (2,2) *+{0} ="145",
    (2,3) *+{0} ="146",
    (2,4) *+{0} ="147",
    (3,3) *+{0} ="156",
    (3,4) *+{0} ="157",
    (4,4) *+{0} ="167",
    (4.5,0) *+{0} ="234",
    (4.5,1) *+{0} ="235",
    (4.5,2) *+{0} ="236",
    (4.5,3) *+{0} ="237",
    (5.5,1) *+{0} ="245",
    (5.5,2) *+{0} ="246",
    (5.5,3) *+{0} ="247",
    (6.5,2) *+{0} ="256",
    (6.5,3) *+{0} ="257",
    (7.5,3) *+{0} ="267",
    (8,0) *+{0} ="345",
    (8,1) *+{0} ="346",
    (8,2) *+{0} ="347",
    (9,1) *+{0} ="356",
    (9,2) *+{0} ="357",
    (10,2) *+{0} ="367",
    (11,0) *+{0} ="456",
    (11,1) *+{0} ="457",
    (12,1) *+{0} ="467",
    (14,0) *+{0} ="567",
    "124", {\ar"123"},
    "125", {\ar"124"},
    "126", {\ar"125"},
    "127", {\ar"126"},
    "134", {\ar"124"},
    "135", {\ar"125"},
    "136", {\ar"126"},
    "137", {\ar"127"},
    "135", {\ar"134"},
    "136", {\ar"135"},
    "137", {\ar"136"},
    "145", {\ar"135"},
    "146", {\ar"136"},
    "147", {\ar"137"},
    "234", {\ar"134"},
    "235", {\ar"135"},
    "236", {\ar"136"},
    "237", {\ar"137"},
    "146", {\ar"145"},
    "147", {\ar"146"},
    "156", {\ar"146"},
    "157", {\ar"147"},
    "245", {\ar"145"},
    "246", {\ar"146"},
    "247", {\ar"147"},
    "157", {\ar"156"},
    "167", {\ar"157"},
    "256", {\ar"156"},
    "257", {\ar"157"},
    "267", {\ar"167"},
    "235", {\ar"234"},
    "236", {\ar"235"},
    "237", {\ar"236"},
    "245", {\ar"235"},
    "246", {\ar"236"},
    "247", {\ar"237"},
    "246", {\ar"245"},
    "247", {\ar"246"},
    "256", {\ar"246"},
    "257", {\ar"247"},
    "345", {\ar"245"},
    "346", {\ar"246"},
    "347", {\ar"247"},
    "257", {\ar"256"},
    "267", {\ar"257"},
    "356", {\ar"256"},
    "357", {\ar"257"},
    "367", {\ar"267"},
    "346", {\ar"345"},
    "347", {\ar"346"},
    "357", {\ar"356"},
    "356", {\ar"346"},
    "357", {\ar"347"},
    "367", {\ar"357"},
    "456", {\ar"356"},
    "457", {\ar"357"},
    "467", {\ar"367"},
    "457", {\ar"456"},
    "467", {\ar"457"},
    "567", {\ar"467"},
    %
    \end{xy}}
    \end{equation*}
    \end{Eg}

 We state the main theorem and illustrate it with an example. The proof will be given in the next section.
    \begin{Thm}\label{thm_main_construction_higher_auslander_alg_typeA}
        Let $A = A_{n+1}^d$ be the $(d-1)$-Auslander algebra of type $\mathbb{A}_{n+1}$ with $\gcd(n, d) = 1$. There is a tilting complex $T = \bigoplus_{i=1}^{n+d}\nu^iP$ where $P$ is the basic projective module with $|P| = \frac{1}{n+d}\binom{n+d}{d}$ defined in Definition \ref{def_the_proj_module}. Denote by $B = \End_{\mathcal{D}^b(A)}(T)$ the endomorphism algebra of $T$. Then the following statements hold.  
        \begin{itemize}
            \item [(i)] There is a triangle equivalence $F: \mathcal{D}^b(A) \xrightarrow[]{\sim} \mathcal{D}^b(B)$ which restricts to an equivalence between the $nd\mathbb{Z}$-cluster tilting subcategories $F: \mathcal{U}_{nd}(T)\xrightarrow{\sim} \mathcal{U}_{nd}(B)$.
            \item [(ii)] $B$ is $2$-subhomogeneous $nd$-representation finite, i.e. $\gldim B = nd$ and $B\oplus DB$ is an $nd\mathbb{Z}$-cluster tilting module.
            \item [(iii)] $B$ is $\frac{nd}{n+d+1}$-CY.
        \end{itemize}
    \end{Thm}

    \begin{Eg}\label{eg_A_5_3_B}
        Let $A = A_5^3$. The tilting complex in Theorem \ref{thm_main_construction_higher_auslander_alg_typeA} is given by $T = \bigoplus_{i=1}^{7}\nu^iP$ where $P = \bigoplus_{j\in J}P_j$. Here $P_j$ is the indecomposable projective module at vertex $j$ and $J=\{(1,2,3), (1,2,4), (1,2,5), (1,3,4), (1,3,5)\}$. 
        The Gabriel quiver of $B = \End_{\mathcal{D}^b(A)}(T)$ is given below. The relations are given by commutativity relations of squares (blue dashed line) and zero relations (red dashed line). This is a special case of the general description of the algebra $B$ given in Proposition \ref{prop_B_B0}.     
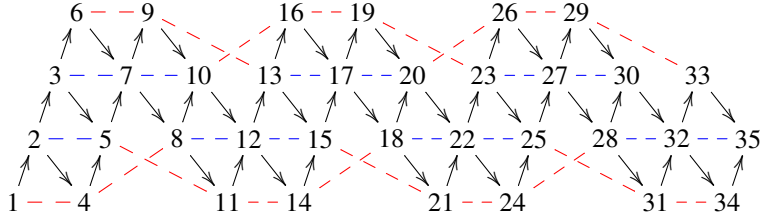
    \begin{equation*}
        \begin{xy}
            0;<18pt,0cm>:<8pt,24pt>::
            (0,0) *+{1} ="1",
            (0,1) *+{2} ="2",
            (0,2) *+{3} ="3",
            (0,3) *+{6} ="6",
            (1.5,0) *+{4} ="4",
            (1.5,1) *+{5} ="5",
            (1.5,2) *+{7} ="7",
            (1.5,3) *+{9} ="9",
            (3,1) *+{8} ="8",
            (3,2) *+{10} ="10",
            (4.5,0) *+{11} ="11",
            (4.5,1) *+{12} ="12",
            (4.5,2) *+{13} ="13",
            (4.5,3) *+{16} ="16",
            (6,0) *+{14} ="14",
            (6,1) *+{15} ="15",
            (6,2) *+{17} ="17",
            (6,3) *+{19} ="19",
            (7.5,1) *+{18} ="18",
            (7.5,2) *+{20} ="20",
            (9,0) *+{21} ="21",
            (9,1) *+{22} ="22",
            (9,2) *+{23} ="23",
            (9,3) *+{26} ="26",
            (10.5,0) *+{24} ="24",
            (10.5,1) *+{25} ="25",
            (10.5,2) *+{27} ="27",
            (10.5,3) *+{29} ="29",
            (12,1) *+{28} ="28",
            (12,2) *+{30} ="30",
            (13.5,0) *+{31} ="31",
            (13.5,1) *+{32} ="32",
            (13.5,2) *+{33} ="33",
            (15,0) *+{34} ="34",
            (15,1) *+{35} ="35",
            "1", {\ar "2"},
            "2", {\ar "3"},
            "3", {\ar "6"},
            "2", {\ar "4"},
            "3", {\ar "5"},
            "6", {\ar "7"},
            "4", {\ar "5"},
            "5", {\ar "7"},
            "7", {\ar "9"},
            "7", {\ar "8"},
            "9", {\ar "10"},
            "8", {\ar "10"},
            "8", {\ar "11"},
            "10", {\ar "12"},
            "11", {\ar "12"},
            "12", {\ar "13"},
            "13", {\ar "16"},
            "12", {\ar "14"},
            "13", {\ar "15"},
            "16", {\ar "17"},
            "14", {\ar "15"},
            "15", {\ar "17"},
            "17", {\ar "19"},
            "17", {\ar "18"},
            "18", {\ar "20"},
            "19", {\ar "20"},
            "18", {\ar "21"},
            "20", {\ar "22"},
            "21", {\ar "22"},
            "22", {\ar "23"},
            "23", {\ar "26"},
            "22", {\ar "24"},
            "23", {\ar "25"},
            "26", {\ar "27"},
            "24", {\ar "25"},
            "25", {\ar "27"},
            "27", {\ar "29"},
            "27", {\ar "28"},
            "29", {\ar "30"},
            "28", {\ar "30"},
            "28", {\ar "31"},
            "30", {\ar "32"},
            "31", {\ar "32"},
            "32", {\ar "33"},
            "32", {\ar "34"},
            "33", {\ar "35"},
            "34", {\ar "35"},
            "1", {\ar@[red]@{--} "4"},
            "6", {\ar@[red]@{--} "9"},
            "11", {\ar@[red]@{--} "14"},
            "16", {\ar@[red]@{--} "19"},
            "21", {\ar@[red]@{--} "24"},
            "26", {\ar@[red]@{--} "29"},
            "31", {\ar@[red]@{--} "34"},
            "4", {\ar@[red]@{--} "8"},
            "5", {\ar@[red]@{--} "11"},
            "9", {\ar@[red]@{--} "13"},
            "10", {\ar@[red]@{--} "16"},
            "14", {\ar@[red]@{--} "18"},
            "15", {\ar@[red]@{--} "21"},
            "19", {\ar@[red]@{--} "23"},
            "20", {\ar@[red]@{--} "26"},
            "24", {\ar@[red]@{--} "28"},
            "25", {\ar@[red]@{--} "31"},
            "29", {\ar@[red]@{--} "33"},
            "2", {\ar@[blue]@{--} "5"},
            "3", {\ar@[blue]@{--} "7"},
            "7", {\ar@[blue]@{--} "10"},
            "8", {\ar@[blue]@{--} "12"},
            "12", {\ar@[blue]@{--} "15"},
            "13", {\ar@[blue]@{--} "17"},
            "17", {\ar@[blue]@{--} "20"},
            "18", {\ar@[blue]@{--} "22"},
            "22", {\ar@[blue]@{--} "25"},
            "23", {\ar@[blue]@{--} "27"},
            "27", {\ar@[blue]@{--} "30"},
            "28", {\ar@[blue]@{--} "32"},
            "32", {\ar@[blue]@{--} "35"},
        \end{xy}
    \end{equation*}
    \end{Eg}

\subsection{Lattice paths}

To prove Theorem \ref{thm_main_construction_higher_auslander_alg_typeA}, we introduce a combinatorial model for $A = A_{n+1}^d$ and $\mathcal{M} = \mathcal{M}_{n+1}^d$ by lattice paths.

For a $(d\times n)$-rectangle, we fix the coordinate as in Figure \ref{fig:fix_coord_of_rectangle}. By a lattice path, we mean a sequence of steps of length $1$ in the direction either $(1,0)$ or $(0,1)$ that 
 connects $(0,0)$ and $(d,n)$.
 
\begin{figure}[h]
    \begin{tikzpicture}[scale=0.6]
        \draw[] (-7,0) rectangle +(5,7);
        \node[] at (-7,-0.4) {$(0,0)$}; 
        \node[] at (-7,7.4) {$(0,n)$}; 
        \node[] at (-2,-0.4) {$(d,0)$};
        \node[] at (-2,7.4) {$(d,n)$};
    \end{tikzpicture}
    \caption{$(d\times n)$-rectangle}
    \label{fig:fix_coord_of_rectangle}
\end{figure}
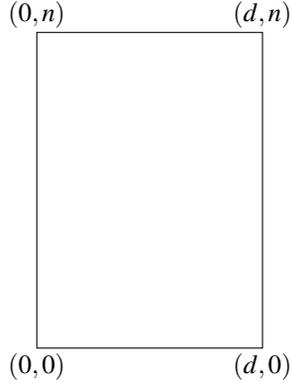

We denote by $L_{d,n}$ the set of all lattice paths in the $(d\times n)$-rectangle. 
This is a set with $\binom{d+n}{d}$ elements.
We introduce coordinates for each element in $L_{d,n}$. Starting from $(0,0)$, we label each step by $1,2,3\ldots, d+n$ and the coordinates of the path are given by the ordered sequence of the labels of horizontal steps. 

\begin{Eg}
    Let $d=3$ and $n=4$. Then the coordinates of the following path is $(1,4,7)$.
    
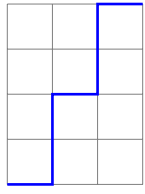
\begin{center}
    \begin{tikzpicture}[scale=0.6]
        (0,0) rectangle +(3,4);
        \draw[help lines] (0,0) grid +(3,4);
        \coordinate (prev) at (0,0);
        \draw [color=blue, line width=1] (0,0)--(1,0)--(1,2)--(2,2)--(2,4)--(3,4);
    \end{tikzpicture}
    \end{center}
\end{Eg}

Note that the coordinates of a lattice path form a tuple of $d$ positive integers which are strictly increasing.
Indeed there is a bijection from $L_{d,n}$ to $\os_{n+1}^d$ by sending each path to its coordinates. We denote this map by 
\begin{align*}
    c : L_{d,n} & \rightarrow \os_{n+1}^d\\
    \ell & \mapsto c(\ell),
\end{align*}
with the inverse map by
\begin{align*}
    l: \os_{n+1}^d & \rightarrow L_{d,n}\\
    x & \mapsto l(x).
\end{align*}

The set $L_{d,n}$ can be naturally viewed as a poset. A lattice path $\ell_1$ is smaller than another $\ell_2$ if $\ell_1$ lies below $\ell_2$, denoted by $\ell_1 < \ell_2$.

The area between a lattice path $\ell$ and the right and bottom boundaries gives a Young diagram $Y_{\ell}$. We denote the set by $Y_{d,n}=\{Y_{\ell}\mid \ell\in L_{d,n}\}$. 
We define a relation $R$ on $L_{d,n}$ as follows: 
$\ell_1 R \ell_2$ if and only if $\ell_1 \leq \ell_2$ and no $(2\times 1)$-rectangles fit in $Y_{\ell_2}-Y_{\ell_1}$. 

We illustrate the relation $R$ with Figure \ref{fig:relation_R}. Here we have $\ell_1<\ell_2<\ell_3$ where $\ell_1R\ell_2$ and $\ell_2R\ell_3$ but $\ell_1R\ell_3$ doesn't hold.

\begin{figure}[h]
    \centering
    \begin{tikzpicture}[scale=0.6]
        (0,0) rectangle +(3,4);
        \draw[help lines] (0,0) grid +(3,4);
        \coordinate (prev) at (0,0);
        \draw [color=blue, line width=1] (0,0)--(0,2)--(1,2)--(1,3)--(2,3)--(2,4)--(3,4);
        \draw [color=green, line width=1] (0,0.05)--(1,0.05)--(1,2)--(2,2)--(2,3)--(2.95,3)--(2.95,4);
        \draw[color=red, line width=1] (0,0)--(2,0)--(2,1)--(2,2)--(3,2)--(3,4);
        \node[] at (1.4, 1.5) {$\ell_2$};
        \node[] at (2.4, 0.5) {$\ell_1$};
        \node[] at (0.6, 2.5) {$\ell_3$};
    \end{tikzpicture}
    \caption{$\ell_1<\ell_2<\ell_3$}
    \label{fig:relation_R}
\end{figure}
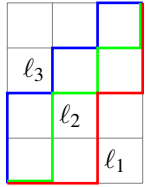

\begin{Prop}
    Let $\ell_1, \ell_2\in L_{d,n}$. We have that 
    $\ell_1 R \ell_2$ if and only if $c(\ell_1)\preccurlyeq c(\ell_2)$.
\end{Prop}
\begin{proof}
    Assume $\ell_1 < \ell_2$.
    Let $c(\ell_1) = (x_1, x_2, \ldots, x_d)$ and $c(\ell_2) = (y_1, y_2, \ldots, y_d)$. We have $x_j\leq y_j$ for all $1\leq j\leq d$.
    Suppose the $(2\times 1)$-rectangle with bottom-left corner $(i,j)$ fits  in $Y_{\ell_2}-Y_{\ell_1}$. 
    Then the following condition holds.
    \begin{equation*}
        x_{i+2}\leq i+j, y_{i+1}\geq i+j+1.
    \end{equation*}
    This implies that $c(\ell_1)\npreceq c(\ell_2)$.

    On the other hand, $c(\ell_1)\npreceq c(\ell_2)$ implies that there exists $1\leq i\leq d-1$ such that $x_{i+1} > y_i$. Then a $(2\times 1)$-rectangle with bottom-left corner $(i-1, x_{i+1})$ would fit into $Y_{\ell_2}-Y_{\ell_1}$.
\end{proof}

Thus there are inverse bijections 
\begin{equation*}
    \xymatrix{
    (L_{d,n}, R)\ar@/^/[r]^{c} & (\os_{n+1}^d, \preccurlyeq)\ar@/^/[l]^{l}
    }
\end{equation*}
Since the vertices in the quiver $Q^{n+1,d}$ of $A$ are labelled by $\os_{n+1}^d$, we may instead label them by $L_{d,n}$. 
The corresponding indecomposable projective and injective $A$-modules will be labelled as $P_{\ell}$ and $I_{\ell}$ respectively for $\ell\in L_{d,n}$.
Similarly the indecomposables in $\mathcal{M}$ are $M(x)$ where $x\in \os_{n+1}^{d+1}$, and so we shall label $M(x) = M_{\ell (x)}$ instead.

Define 
\begin{align*}
    \bar{\cdot}:  L_{d,n} & \rightarrow  L_{d+1,n} & \ell & \mapsto \bar{\ell} \\
    \tilde{\cdot}:  L_{d,n} & \rightarrow L_{d+1,n} & \ell & \mapsto \tilde{\ell},
\end{align*}
where $\bar{\ell}$ (respectively $\tilde{\ell}$) is obtained from $\ell$ by adding a horizontal step at the beginning (respectively the end). 
We reinterpret parts of Proposition \ref{prop_A_n_d} in terms of lattice paths. See Example \ref{eg_lattice_path_proj_inj} and \ref{eg_(p,q)_strip} for specific cases.

\begin{Prop}\label{prop_M_path}
    The following statements hold.
    \begin{itemize}
        \item [(i)] Let $p\in L_{d+1, n}$. Then $M_p$ is projective (respectively injective) if and only if the first (respectively last) step of $p$ is horizontal.
        \item [(ii)] We have $P_{\ell} = M_{\bar{\ell}}$ and $I_{\ell} = M_{\tilde{\ell}}$ for all $\ell\in L_{d,n}$. 
        \item [(iii)] Let $M_p, M_q\in \mathcal{M}$ such that there is a minimal $(d+2)$-exact sequence 
        \begin{equation*}
            \xymatrix@C=1em@R=1em{
        0\ar[r] & M_p\ar[r] & M_{\ell_d}\ar[r] & \cdots\ar[r] & M_{\ell_i}\ar[r] & \cdots\ar[r] & M_{\ell_1}\ar[r] & M_q\ar[r] & 0.}
        \end{equation*}
        Then $q$ can be obtained from $p$ by shifting $1$ step along the vector $(-1, 1)$ such that the strip bounded by $p$ and $q$ has width $1$. We refer to it as $(p,q)$-strip. The middle term $\ell_i$ is a lattice path in which the first $i$ horizontal steps coincide with that of $p$ and the last $d+1-i$ horizontal steps coincide with that of $q$.
    \end{itemize}
\end{Prop}

\begin{proof}
    Write $M_p = M(c(p))$ where $c(p) = (c_1, \ldots, c_{d+1})$. By Proposition \ref{prop_A_n_d} $(i)$, $M(c(p))$ is projective (respectively injective) if and only if $c_1 = 1$ (respectively $c_{d+1} = n+d+1$) which is equivalent to statement $(i)$.

    Let $x=(x_1, \ldots, x_{d+1})$. We have $\top M(x) = S_{(x_2-1, \ldots, x_{d+1}-1)}$ and $\soc M(x) = S_{(x_1, \ldots, x_d)}$ where $S_y$ denotes the simple module at vertex $y\in \os_{n+1}^d$. 
    Thus $\top M_{\bar{\ell}} = S_{\ell}$.
    Moreover, the first step of $\bar{\ell}$ being horizontal implies that $M_{\bar{\ell}}$ is projective. So  $M_{\bar{\ell}} = P_{\ell}$. The statement for $I_{\ell}$ can be shown similarly. So statement $(ii)$ follows.

    Write $M_q = M(c(q))$ and $M_{\ell_i} = M(c(\ell_i))$ where $c(q) = (c'_1, \ldots, c'_{d+1})$ and $c(\ell_i) = (f_1^i, \ldots, f_{d+1}^i)$ for $1\leq i\leq d$. 
    By Proposition \ref{prop_A_n_d} $(iv)$, we have 
    \begin{equation*}
        f_j^i = \left\{ \begin{array}{ll}
                c_j &  j\leq i\\
                c'_j & j > i
                \end{array} \right.
    \end{equation*}
    which is equivalent to statement $(iii)$.
\end{proof}

We illustrate the statements $(ii)$ and $(iii)$ in Proposition \ref{prop_M_path} with the following examples.

\begin{Eg}\label{eg_lattice_path_proj_inj}
Let $\ell=l(1,3,5)\in L_{3,4}$. We have $P_{\ell} = M_{\bar{\ell}}$ where $c(\bar{\ell}) = (1,2,4,6)$ and $I_{\ell} = M_{\tilde{\ell}}$ where $c(\tilde{\ell}) = (1,3,5,8)$.

\begin{center}
\begin{tikzpicture}[scale=0.6]
(-5,0) rectangle +(3,4);
\draw[help lines] (-5,0) grid +(3,4);
\draw [color=blue, line width=1] (-5,0)--(-4,0)--(-4,1)--(-3,1)--(-3,2)--(-2,2)--(-2,4);
\node [] at (-3.5, -0.8) {$\ell$};
(1,0) rectangle +(4,4);
\draw[help lines] (1,0) grid +(4,4);
\coordinate (prev) at (1,0);
\draw [color=blue, line width=1] (1,0)--(2,0)--(3,0)--(3,1)--(4,1)--(4,2)--(5,2)--(5,4);
\node [] at (3, -0.8) {$\bar{\ell}$};
(8,0) rectangle +(4,4);
\draw[help lines] (8,0) grid +(4,4);
\draw [color=blue, line width=1] (8,0)--(9,0)--(9,1)--(10,1)--(10,2)--(11,2)--(11,4)--(12,4);
\node [] at (10, -0.8) {$\tilde{\ell}$};
\end{tikzpicture}
\end{center}
\end{Eg}

\begin{Eg}\label{eg_(p,q)_strip} 
Let $p$ be the blue path and $q$ be the green one. The $(p,q)$-strip shown in the picture gives the following  minimal $5$-exact sequence.
\begin{equation*}
    \xymatrix@C=1em@R=1em{0\ar[r] & M_p\ar[r] & M_{\ell_3}\ar[r] & M_{\ell_2}\ar[r] & M_{\ell_1}\ar[r] & M_q\ar[r] & 0.}
\end{equation*}
   
    
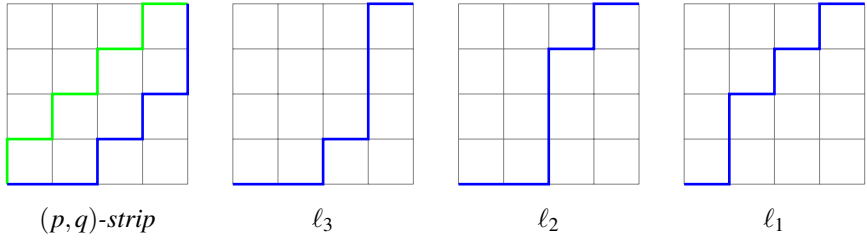
\begin{tikzpicture}[scale=0.6]
        (-5,0) rectangle +(4,4);
        \draw[help lines] (-5,0) grid +(4,4);
        \coordinate (prev) at (-5,0);
        \draw [color=green, line width=1] (-5,0)--(-5,1)--(-4,1)--(-4,2)--(-3,2)--(-3,3)--(-2,3)--(-2,4)--(-1,4);
        \draw [color=blue, line width=1] (-5,0)--(-4,0)--(-3,0)--(-3,1)--(-2,1)--(-2,2)--(-1,2)--(-1,4);
        \node [] at (-3, -0.8) {$(p,q)$-strip};
        (0,0) rectangle +(4,4);
        \draw[help lines] (0,0) grid +(4,4);
        \coordinate (prev) at (0,0);
        \draw [color=blue, line width=1] (0,0)--(1,0)--(2,0)--(2,1)--(3,1)--(3,4)--(4,4);
        \node[] at (2, -0.8) {$\ell_3$};
        (5,0) rectangle +(4,4);
        \draw[help lines] (5,0) grid +(4,4);
        \draw [color=blue, line width=1] (5,0)--(6,0)--(7,0)--(7,3)--(8,3)--(8,4)--(9,4);
        \node[] at (7, -0.8) {$\ell_2$};
        (10,0) rectangle +(4,4);
        \draw[help lines] (10,0) grid +(4,4);
        \draw [color=blue, line width=1] (10,0)--(11,0)--(11,2)--(12,2)--(12,3)--(13,3)--(13,4)--(14,4);
        \node[] at (12, -0.8) {$\ell_1$};
    \end{tikzpicture}
\end{Eg}

Recall the $d\mathbb{Z}$-cluster tilting subcategory in $\mathcal{D}^b(A)$ induced by $\mathcal{M}$.
\begin{equation*}
    \mathcal{U} = \{\nu_d^i(A)\mid i\in \mathbb{Z}\} = \add \{M_{\ell}[di]\mid \ell\in L_{n+1,d}, i\in \mathbb{Z}\}.
\end{equation*}
Thus an indecomposable object in $\mathcal{U}$ can be represented by $\ell\in L_{d+1,n}$ decorated by a shift $[di]$ for $i\in \mathbb{Z}$. 
Since $\mathcal{U}[d] = \mathcal{U}$, $\mathcal{U}$ is closed under the Nakayama functor $\nu$. 
We now describe the $\nu$-orbit of an arbitrary  object in $\mathcal{U}$. 
To do this, we define the following rotation map of a lattice path.
\begin{align*}
    r :  L_{d, n} & \rightarrow  L_{d, n} \\
    \ell & \mapsto r(\ell)
\end{align*}
where $r(\ell)$ is obtained from $\ell$ by moving the first step to the last. Note that this induces a $\mathbb{Z}/(n+d)\mathbb{Z}$-action on $L_{d,n}$. 

\begin{Eg} Let $\ell=l(1,3,5)\in L_{3,4}$. We have $r(\ell) = l(2,4,7)\in L_{3,4}$.
    
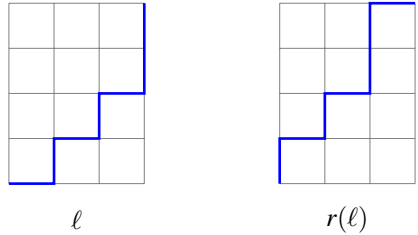
\begin{center}
\begin{tikzpicture}[scale=0.6]
(-5,0) rectangle +(3,4);
\draw[help lines] (-5,0) grid +(3,4);
\draw [color=blue, line width=1] (-5,0)--(-4,0)--(-4,1)--(-3,1)--(-3,2)--(-2,2)--(-2,4);
\node [] at (-3.5, -0.8) {$\ell$};
(1,0) rectangle +(3,4);
\draw[help lines] (1,0) grid +(3,4);
\coordinate (prev) at (1,0);
\draw [color=blue, line width=1] (1,0)--(1,1)--(2,1)--(2,2)--(3,2)--(3,4)--(4,4);
\node [] at (2.5, -0.8) {$r({\ell})$};
\end{tikzpicture}
\end{center}
\end{Eg}

\begin{Prop}\label{prop_nu_orbit_M}
Let $M_{\ell}[di]\in \mathcal{U}$. We have
    \begin{equation*}
    \nu (M_{\ell}[di]) = 
    \left\{ \begin{array}{ll}
                M_{r(\ell)}[di] &  \text{ if the first step of $\ell$ is horizontal }\\
                M_{r(\ell)}[d(i+1)] & \text{ otherwise.}
                \end{array} \right.
\end{equation*}
\end{Prop}
\begin{proof}
    Write $c({\ell}) = (c_1, \ldots, c_{d+1})$.
    If the first step of $\ell$ is horizontal, then $c_1 = 1$ which implies that $M_{\ell}$ is projective. 
    Combining Proposition \ref{prop_A_n_d} $(i)$ and $(ii)$, we have
    \begin{equation*}
        \nu M_{\ell} = \nu M(1, c_2, \ldots, c_{d+1}) = M(c_2-1, \ldots, c_{d+1}-1, n+d+1) = M_{r(\ell)}.
    \end{equation*}
    Otherwise, by definition, 
    \begin{equation*}
        M_{\ell} = \nu_d^j(P), P\in \add A, \text{ for some } j<0.
    \end{equation*}
    So $\nu_d(M_{\ell}) = \tau_d(M_{\ell})$ by Theorem \ref{thm_d_rf_alg} (a). 
    Moreover $\tau_d M_{\ell} = M(c_1-1, \ldots, c_{d+1}-1)$ by Proposition \ref{prop_A_n_d} (ii). 
    We have 
    \begin{equation*}
        \nu M_{\ell}  = \nu_d M_{\ell}[d] = \tau_d M_{\ell}[d] = M_{r(\ell)}[d].
    \end{equation*}
\end{proof}

\subsection{Rational Dyck paths}
From this section on, we assume that $\gcd(n,d)=1$.
A rational $(d,n)$-Dyck path is a lattice path  in the $(d\times n)$-rectangle that stays below and never crosses the diagonal. We denote by $\Dyck_{d,n}$ the set of rational Dyck paths.

\begin{Eg}
    Let $d=3$ and $n=4$. $\Dyck_{3,4}$ consists of the following lattice paths.
    
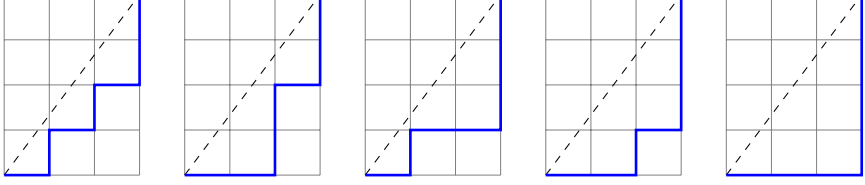
\begin{tikzpicture}[scale=0.6]
(0,0) rectangle +(3,4);
\draw[help lines] (0,0) grid +(3,4);
\draw[dashed] (0,0) -- +(3,4);
\coordinate (prev) at (0,0);
\draw [color=blue, line width=1] (0,0)--(1,0)--(1,1)--(2,1)--(2,2)--(3,2)--(3,4);

(4,0) rectangle +(3,4);
\draw[help lines] (4,0) grid +(3,4);
\draw[dashed] (4,0) -- +(3,4);
\coordinate (prev) at (4,0);
\draw [color=blue, line width=1] (4,0)--(5,0)--(6,0)--(6,2)--(7,2)--(7,4);

(8,0) rectangle +(3,4);
\draw[help lines] (8,0) grid +(3,4);
\draw[dashed] (8,0) -- +(3,4);
\coordinate (prev) at (8,0);
\draw [color=blue, line width=1] (8,0)--(9,0)--(9,1)--(10,1)--(11,1)--(11,4);

(12,0) rectangle +(3,4);
\draw[help lines] (12,0) grid +(3,4);
\draw[dashed] (12,0) -- +(3,4);
\coordinate (prev) at (12,0);
\draw [color=blue, line width=1] (12,0)--(13,0)--(14,0)--(14,1)--(15,1)--(15,4);

(16,0) rectangle +(3,4);
\draw[help lines] (16,0) grid +(3,4);
\draw[dashed] (16,0) -- +(3,4);
\coordinate (prev) at (16,0);
\draw [color=blue, line width=1] (16,0)--(17,0)--(18,0)--(19,0)--(19,4);

\end{tikzpicture}
\end{Eg}

Recall the rotation map $\ell\mapsto r(\ell)$ gives a $\mathbb{Z}/(d+n)\mathbb{Z}$ action on $L_{d,n}$. It is a free action thus each orbit contains $n+d$ elements. Indeed, there is exactly one rational Dyck path in each orbit \cite[Page 57]{Biz54}. Thus there are $\frac{1}{d+n}\binom{d+n}{d}$ elements in $\Dyck_{d,n}$.

Let $\ell\in \Dyck_{d,n}$. Recall that $\bar{\ell}\in L_{d+1, n}$. Now we study the $r$-orbit of $\bar{\ell}$. 

For each lattice point $(x,y)$ with $x\leq d$ in the $((d+1)\times n)$-rectangle, we consider the curve $\delta_{(x,y)}$ which consists of a horizontal step from $(x,y)$ to $(x+1, y)$ together with the rays to $(x,y)$ and from $(x+1, y)$ of slope $\frac{n}{d}$. See Figure \ref{fig:delta_curve_region} (a) for  a generic example of such a curve $\delta_{(x,y)}$.

\begin{figure}[h]
\begin{center}
    \begin{tikzpicture}[scale=0.6]
    \draw[] (-7,0) rectangle +(5,7);
\node[] at (-7,-0.4) {$(0,0)$}; 
\node[] at (-7.8,7) {$(0,n)$}; 
\node[] at (-2,-0.4) {$(d+1,0)$};
\node[] at (-4.5, -0.8) {$(a)$};
\draw [color=red, line width=1] (-7,1.5)--(-5,5)--(-4,5)--(-3.14,7);
\node [] at (-5, 5.4) {$(x,y)$};
\node [] at (-2.7, 6) {$\delta_{(x,y)}$};
\draw[] (2,0) rectangle +(5,7);
\coordinate (prev) at (2,0);
\fill[blue!20] (2,1.5)--(2,0)--(4,0)--(4,5)--cycle;
\fill[blue!20] (5,5)--(6.14,7)--(7,7)--(7,5)--cycle;
\draw [color=red, line width=1] (2,1.5)--(4,5)--(5,5)--(6.14,7);
\node [] at (4, 5.4) {$(x,y)$};
\node [] at (3.3, 1.8) {$R_{(x,y)}$};
\draw [color=blue, line width=1] (2,0)--(4,0)--(4,5)--(7,5)--(7,7);
\node [] at (4.8, 2.5) {$\ell_{(x,y)}$};
\node [] at (4.5, -0.8) {$(b)$};
    \end{tikzpicture}
\end{center}
    \caption{$\delta_{(x,y)}$ and $R_{(x,y)}$}
    \label{fig:delta_curve_region}
\end{figure}
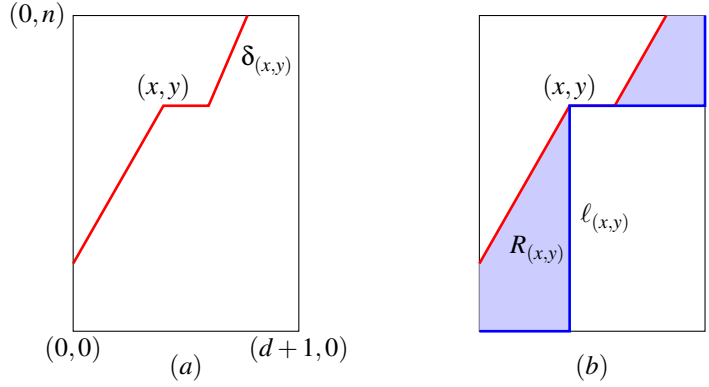
For such a curve $\delta_{(x,y)}$, we introduce the region $R_{(x,y)}\subseteq L_{d+1,n}$. We consider the lattice path $\ell_{(x,y)} = l(1,2,\ldots, x, x+y, x+y+1, \ldots, d+y)$ and define
\begin{equation*}
    R_{(x,y)} = \{\ell\in L_{d+1,n}\mid \ell\geq \ell_{(x,y)} \text{ and } \ell \text{ lies below } \delta_{(x,y)}\}.
\end{equation*}
See Figure \ref{fig:delta_curve_region} (b)  for an example of $\ell_{(x,y)}$ and $R_{(x,y)}$. 

Note that $\ell\in \Dyck_{d,n}$ precisely means that $\bar{\ell}\in R_{(0,0)}$. 
Suppose $\bar{\ell}$ passes through a lattice point $D = (x,y)$. The fact that $D$ lies below $\delta_{(0,0)}$ implies 
\begin{equation*}
    (x-1)\frac{n}{d} \leq y.
\end{equation*}
This is equivalent to 
\begin{equation*}
    (d+1-x)\frac{n}{d} \geq n-y,
\end{equation*}
which in turn implies that the lattice point  $D' = (d+1-x, n-y)$ lies above $\delta_{(d,n)}$.

Let 
\begin{equation*}
    \Delta = \{D\mid D \neq (d+1,n) \text{ is a lattice point that lies below } \delta_{(0,0)}\} = \bigcup_{i=0}^{n+d} \Delta_i
\end{equation*}
where $\Delta_i = \{(x,y)\in \Delta \mid x+y=i\}$ and 
\begin{equation*}
    \Delta' = \{D\mid D\neq (0,0) \text{ is a lattice point that lies above } \delta_{(d,n)}\} = \bigcup_{i=0}^{n+d} \Delta_i'
\end{equation*}
where $\Delta_i' = \{(x,y)\in \Delta'\mid d+1-x+n-y=i\}$.
By the discussion above, we have that there is a bijection from $\Delta$ to $\Delta'$ by sending $D= (x,y)$ to $D'=(d+1-x, n-y)$.

For each $D\in \Delta$, we introduce the following subset $S_D\subseteq R_{(0,0)}$. See Figure \ref{fig:region_S_D} for an example of such a regoin $S_D$.
\begin{equation*}
    S_D =\{\ell\in R_{(0,0)}\mid \ell \text{ passes through } D \}.
\end{equation*}
\begin{Rem}
    For any $0\leq i\leq n+d$, each $\ell\in R_{(0,0)}$ lies in $S_D$ for some unique $D\in \Delta_i$.
\end{Rem}

\begin{figure}[h]
    \centering
    \begin{tikzpicture}[scale=0.6]
        \draw[] (0,0) rectangle +(5,7);
        \draw[color=red, line width=1] (0,0)--(1,0)--(5,7);
        \draw[color=blue, line width=1] (0,2)--(5,2);
        \draw[color=blue, line width=1] (3,0)--(3,7);
        \fill[orange!30] (1,0)--(3,0)-- (3,2)--(2.14,2)--cycle;
        \fill[orange!30] (3,2)--(5,2)--(5,7)--(3,3.5)--cycle;
        \draw (3,2) node [scale=0.3, circle, draw,fill=red]{};
        \node[] at (3.4, 1.7) {$D$};
        \node[] at (4, 3) {$S_D$};
    \end{tikzpicture}
    \caption{the region $S_D$ for $D\in \Delta$}
    \label{fig:region_S_D}
\end{figure}
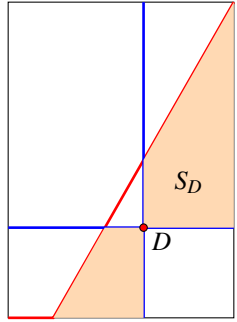

\begin{Lem}\label{lem_bijection_S_D_R_D'}
    For $1\leq i\leq n+d$ and $D=(x,y)\in \Delta_i$, we have $D'=(d+1-x, n-y)\in \Delta'_i$ and 
    \begin{align*}
        r^i: S_D\rightarrow R_{D'} 
    \end{align*}
    is bijective.
\end{Lem}
\begin{proof}
    Let $\ell\in S_D$. Let $\ell^{\ast}$ be the path from $(0,0)$ to $(2d+2, 2n)$ obtained by extending $\ell$ periodically. 
    Now $r^i(\ell)$ is obtained by taking the subpath of $\ell^{\ast}$ from $D=(x,y)$ to $E=(x+d+1, y+n)$ and translating $(x,y)$ to $(0,0)$, see Figure \ref{fig:ell_r_i_ell} for an illustration.  
    Thus $r^i(\ell)\in R_{D'}$.
    This provides a bijective correspondence.
    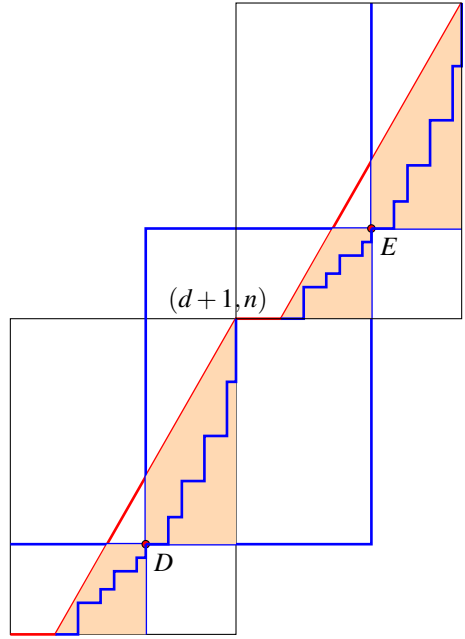
\begin{figure}[h]
        \centering
        \begin{tikzpicture}[scale=0.6]
            \draw[] (0,0) rectangle +(5,7);
            \draw[color=red, line width=1] (0,0)--(1,0)--(5,7)--(6,7)--(10,14);
            \draw[color=blue, line width=1] (0,2)--(8,2)--(8,9)--(3,9)--(3,0);
            \draw[color=blue, line width=1] (8,9)--(10,9);
            \draw[color=blue, line width=1] (8,9)--(8,14);
            \fill[orange!30] (1,0)--(3,0)-- (3,2)--(2.14,2)--cycle;
            \fill[orange!30] (3,2)--(5,2)--(5,7)--(3,3.5)--cycle;
            \fill[orange!30] (6,7)--(8,7)--(8,9)--(7.14,9)--cycle;
            \fill[orange!30] (8,9)--(10,9)--(10,14)--(8,10.5)--cycle;
            \draw (3,2) node [scale=0.3, circle, draw,fill=red]{};
            \draw (8,9) node [scale=0.3, circle, draw,fill=red]{};
            \draw[color=blue, line width=1] (1,0)--(1.5,0)--(1.5,0.7)--(2,0.7)--(2,1)--(2.3,1)--(2.3,1.4)--(2.8,1.4)--(2.8,1.7)--(3,1.7)--(3,2);
            \draw[color=blue, line width=1] (3,2)--(3.5,2)--(3.5,2.6)--(3.8,2.6)--(3.8,3.4)--(4.3,3.4)--(4.3,4.4)--(4.8,4.4)--(4.8,5.6)--(5,5.6)--(5,7);
            \draw[color=blue, line width=1] (6,7)--(6.5,7)--(6.5,7.7)--(7,7.7)--(7,8)--(7.3,8)--(7.3,8.4)--(7.8,8.4)--(7.8,8.7)--(8,8.7)--(8,9);
            \draw[color=blue, line width=1] (8,9)--(8.5,9)--(8.5,9.6)--(8.8,9.6)--(8.8,10.4)--(9.3,10.4)--(9.3,11.4)--(9.8,11.4)--(9.8,12.6)--(10,12.6)--(10,14);
            \node[] at (3.4, 1.6) {$D$};
            \node[] at (8.4, 8.6) {$E$};
            \node[] at (4.6, 7.4) {$(d+1,n)$};
            \draw[] (5,7) rectangle+(5,7); 
        \end{tikzpicture}
        \caption{$\ell$ and $r^i(\ell)$}
        \label{fig:ell_r_i_ell}
    \end{figure}
\end{proof}

The following observation will be useful later.

\begin{Prop}\label{prop_existence_2*1_square}
    Let $\ell, \ell'\in \Dyck_{d,n}$. Then 
    $\bar{\ell}Rr^i(\bar{\ell'})$ doesn't hold for $2\leq i\leq n+d$.
\end{Prop}
\begin{proof}
It suffices to consider the case $\bar{\ell} \leq r^i(\bar{\ell'})$.
We have $\bar{\ell'}\in S_D$ for some $D\in \Delta_i$.
So by Lemma \ref{lem_bijection_S_D_R_D'},  $r^i(\bar{\ell'})\in R_{D'}$ for the corresponding $D'\in \Delta'_i$.
In other words, $r^i(\bar{\ell'})$ passes through $D'$ which is above $\delta_{(d,n)}$. 
As shown in Figure \ref{fig:exists_2_times_1_rectangle}, we  put a $(2\times 1)$-rectangle with $D'$ as the top left corner. 
Since $\bar{\ell}$ lies below $\delta_{(0,0)}$, the rectangle fits in 
$Y_{r^i(\bar{\ell'})} - Y_{\bar{\ell}}$.
So $\bar{\ell}Rr^i(\bar{\ell'})$ doesn't hold.

\begin{figure}[h]
\begin{center}
\begin{tikzpicture}[scale=0.6]
\draw[] (0,0) rectangle +(5,7);
\coordinate (prev) at (0,0);
\draw [color=green, line width=1] (0,0)--(4,7)--(5,7);
\draw [color=orange, line width=1] (0,0)--(1,0)--(5,7);
\draw [color=red, line width=1] (0, 1.5)--(2,5)--(3,5)--(4.14,7);
\draw [color=blue, line width=1] (0,0)--(2,0)--(2,5)--(5,5)--(5,7);
\draw [color=blue, line width=1] (2,4)--(4,4)--(4,5);
\node [] at (2, 5.4) {$D'$};
\node [] at (1.3, 2.3) {$\delta_{(d,n)}$};
\node [] at (2.8, 2.5) {$\delta_{(0,0)}$};
\node [] at (4, 3.6) {$E$};
\end{tikzpicture}
\end{center}
    \caption{the $(2\times 1)$-rectangle}
    \label{fig:exists_2_times_1_rectangle}
\end{figure}
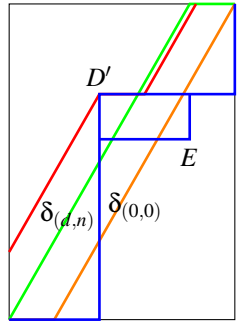

\end{proof}

\subsection{The tilting complex}
In this section, we give the definition of the basic projective module as announced before. 
We introduce a complex in $\mathcal{U}$ and prove that it is tilting.

\begin{Def}\label{def_the_proj_module}
    Let $P = \bigoplus_{\ell\in \Dyck_{d,n}} P_{\ell}$.
\end{Def}

\begin{Prop}\label{prop_nu_orb_P}
    For $1\leq i \leq n+d$, we have 
    \begin{equation*}
        \nu^i P = \bigoplus_{(x,y)\in \Delta'_i}\bigoplus_{\ell\in R_{(x, y)}} M_{\ell}[dy].
    \end{equation*}
\end{Prop}
\begin{proof}
Let $\ell\in \Dyck_{d,n}$.
We have $\bar{\ell}\in S_D$ for some unique $D=(x,y)\in \Delta_i$ and $r^i(\bar{\ell})\in R_{D'}$ with $D'=(d+1-x, n-y)\in \Delta'_i$ by Lemma \ref{lem_bijection_S_D_R_D'}.
Since 
\begin{equation*}
    \nu^i(P_{\ell}) = \nu^iM_{\bar{\ell}} \cong  M_{r^i(\bar{\ell})}[dy]
\end{equation*}
by Proposition \ref{prop_nu_orbit_M},
we obtain the result using the bijection  $r^i:S_D\rightarrow R_{D'}$.   
\end{proof}

\begin{Prop}\label{prop_rigid_P_forall_k}
    We have 
    \begin{equation*}
        \Hom_{\mathcal{D}^b(A)}(P, \nu^iP[k]) = 0
    \end{equation*}
    for all $2\leq i\leq n+d$ and $k\in \mathbb{Z}$.
\end{Prop}
\begin{proof}
    It suffices to show that $\Hom_{\mathcal{D}^b(A)}(P_{\ell}, \nu^iP_{\ell'}[k])=0$ for all $\ell, \ell'\in \Dyck_{d,n}$.
    By Proposition \ref{prop_nu_orb_P}, write $P_{\ell} = M_{\bar{\ell}}$ and $\nu^iP_{\ell'} = M_{r^i(\bar{\ell'})}[sd]$ where $0\leq s\leq i$ is the number of vertical steps of $\ell'$ among the first $i$ steps.
    Since $P_{\ell}$ is a projective module, it is enough to consider when $\nu^iP_{\ell'}[k]\in \modn A$, which implies $k=-sd$. 
    Thus we need to show $\Hom_{\mathcal{D}^b(A)}(M_{\bar{\ell}}, M_{r^i(\bar{\ell'})}) = 0$ for $2\leq i\leq n+d$. 
    By Proposition \ref{prop_existence_2*1_square}, the statement follows.
\end{proof}

\begin{Cor}\label{cor_hom_p_to_nu_power_p}
    We have $\Hom_{\mathcal{D}^b(A)}(P, \nu^iP)=0$ for all $i\geq 2$.
\end{Cor}
\begin{proof}
    Write $i = (n+d+1)q+r$ with $q\geq 0$ and $0\leq r\leq n+d$.
    since $A$ is $\frac{nd}{n+d+1}$-CY by Proposition \ref{prop_A_n_d} (iv), 
    $\nu^{n+d+1} \cong [nd]$. 
    We have 
    \begin{equation*}
        \Hom_{\mathcal{D}^b(A)}(P, \nu^iP) \cong \Hom_{\mathcal{D}^b(A)}(P, \nu^rP[ndq]).
    \end{equation*}
    If $r\geq 2$, then $\Hom_{\mathcal{D}^b(A)}(P, \nu^rP[ndq])=0$ by Proposition \ref{prop_rigid_P_forall_k}.
    If $r\leq 1$, then $\nu^rP\in \modn A$ and $\Hom_{\mathcal{D}^b(A)}(P, \nu^rP[ndq])=0$ as $ndq>0$ and $P\in \proj A$.
\end{proof}

Now we are ready to introduce the main object of this section. Let 
\begin{equation*}
    T=\bigoplus_{i=1}^{d+n}\nu^iP \in \mathcal{U}.
\end{equation*}
The rest of the section is to show that $T$ is actually tilting.
Thus we can apply  Theorem \ref{thm_construction_minimal_dRF_alg} to $X=\nu P$.

\begin{Prop}\label{prop_tilt_rigidity}
    We have $\Hom(T, T[k]) = 0$ for all $k\neq 0$.
\end{Prop}
\begin{proof}
    We have 
    \begin{align*}
        \Hom(T, T[k]) & = \bigoplus_{1\leq i, j\leq d+n} \Hom_{\mathcal{D}^b(A)}(\nu^iP, \nu^jP[k])\\
        & \cong \bigoplus_{1\leq i\leq j\leq d+n}\Hom_{\mathcal{D}^b(A)}(P, \nu^{j-i}P[k]) \oplus \bigoplus_{1\leq j<i\leq d+n}D\Hom_{\mathcal{D}^b(A)}(P, \nu^{i-j+1}P[-k])\\
        & = \bigoplus_{0\leq t\leq d+n-1}\Hom_{\mathcal{D}^b(A)}(P, \nu^{t}P[k]) \oplus \bigoplus_{2\leq t\leq d+n}D\Hom_{\mathcal{D}^b(A)}(P, \nu^{t}P[-k]).
    \end{align*}
    So it suffices to show that $\Hom_{\mathcal{D}^b(A)}(P, \nu^{t}P[k])=0$ for  $0\leq t\leq d+n$ and $k\neq 0$.
    For $t\geq 2$, it follows by Proposition \ref{prop_rigid_P_forall_k}. 
    Since $P$ is projective, 
    \begin{equation*}
        \Hom_{\mathcal{D}^b(A)}(P, M[k]) \cong \Hom_{\mathcal{K}^b(\proj A)}(P, M[k]) = 0
    \end{equation*}
    for $M\in \modn A$ and $k\neq 0$.
    So the statement holds for $t = 0, 1$.
\end{proof}

Next we introduce two functions defined on each $\ell\in L_{d+1,n}$. 
For each lattice point $D$ on $\ell$, consider the line through $D$ of slope $\frac{n}{d}$.
One of these lines is furthest to the left.
Since $\gcd(n,d)=1$, this line determines $D$ uniquely except when the line passes through $(0,0)$ and $(d,n)$ lies on $\ell$.
In that case, we pick $D=(d,n)$.
We call $D$ the anchor of $\ell$.
Now consider $\ell$ in relation to $\delta_D$. We have the following observations. See Figure \ref{fig:ell_anchor} for an illustration.
\begin{itemize}
    \item [(i)] First note $D\in (\Delta'\cup (0,0))\backslash \{(d+1,n)\}$.
    \item [(ii)] The segment of $\ell$ before $D$ lies below $\delta_D$.
    \item [(iii)] In the segment after $D$, there may be several corner points $F$ on $\ell$ that are above $\delta_D$.
    For each such $F$, we denote the horizontal and vertical distance from $F$ to $\delta_D$ by $v_F$ and $w_F$ respectively.
    By the choice of $D$, we have $0<v_F<1$ and so $0<w_F<\frac{n}{d}$.
\end{itemize}

\begin{figure}[h]
\begin{center}
\begin{tikzpicture}[scale=0.6]
\coordinate (prev) at (0,0);
\draw [color=blue, line width=1] (0,0)--(0,2)--(2,2)--(2,3)--(3,3)--(3,7)--(4,7)--(4,9)--(5,9);
\draw [color=red, line width=1] (-1.11,0)--(0,2)--(1,2)--(3,5.6)--(3.77, 7)--(4.5, 8.3)--(4.88, 9);
\draw (0,2) node [scale=0.2, circle, draw,fill=red]{};
\draw (3,7) node [scale=0.2, circle, draw,fill=red]{};
\node[] at (0, 2.4) {$D$};
\node[] at (2.7, 7.2) {$F$};
\node[] at (2.6, 6.5) {$w_F$};
\node[] at (3.5, 7.2) {$v_F$};
\end{tikzpicture}
\end{center}
    \caption{$\ell\in L_{d+1,n}$ with the anchor $D=(x_D, y_D)$}
    \label{fig:ell_anchor}
\end{figure}
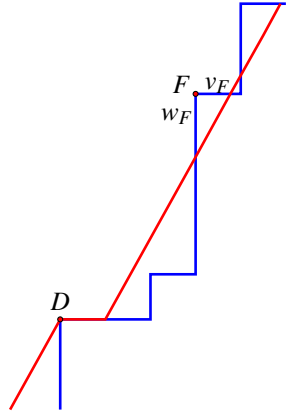

We define $h_{\ell} = y_D\in \mathbb{N}$ and
$\mu_{\ell} = \sum_F w_F^2\in \mathbb{R}$. 
Intuitively one may think that $\mu_{\ell}$ measures the area bounded by $\delta_D$ and the segment of $\ell$ which is above $\delta_D$.
Note that $D=(0,0)$ is equivalent to $h_{\ell}=0$. 
If $D\neq (0,0)$ and $\mu_{\ell}=0$, then $\ell\in R_D$ and since $D\in \Delta'_i$ for some $i$ where $1\leq i\leq n+d$, we get \begin{equation*}
    M_{\ell}\in \add\{\nu^iP[k]\mid k\in \mathbb{Z}\}
\end{equation*}
by Proposition \ref{prop_nu_orb_P}.

To simplify notation, set $\mathbb{M}=L_{d+1,n}$.
We have the following decomposition.
\begin{equation*}
    \mathbb{M} = \bigcup_{h} \mathbb{M}_h = \bigcup_{h}\bigcup_{\mu} \mathbb{M}_{h, \mu},
\end{equation*}
where $\mathbb{M}_{h, \mu} = \{\ell\in \mathbb{M}\mid h_{\ell} = h, \mu_{\ell} = \mu\}$.
Note $\mathbb{M}_n=\mathbb{M}_{n,0}$.

Set
\begin{equation*}
    \mathbb{T} = \bigcup_{h=1}^{n} \mathbb{M}_{h,0}.
\end{equation*}
By the above observations, we have 
$M_{\ell}\in \add\{T[k]\mid k\in\mathbb{Z}\}$ for all $\ell\in \mathbb{T}$.
Moreover, we denote by $\mathbb{M}_{\geq 1} = \bigcup_{h=1}^{n}\mathbb{M}_h$ and $\mathbb{M}'_{\geq 1}=\mathbb{M}_{\geq 1}\backslash \mathbb{T}$.

\begin{Lem}\label{lem_the_resolving_exact_seq}
    For each $\ell\in \mathbb{M}_{\geq 1}'$, there exists a minimal $(d+2)$-exact sequence 
    \begin{equation*}
        \xymatrix@R=1em@C=1em{0\ar[r] & M_{\ell_1}\ar[r] & \cdots\ar[r]  & M_{\ell_j}\ar[r] & M_{\ell}\ar[r] & M_{\ell_{j+1}}\ar[r] & \cdots\ar[r] & M_{\ell_{d+1}}\ar[r] & 0}
    \end{equation*}
    such that $\ell_i\in \bigcup_{h > h_{\ell}}\mathbb{M}_h \cup \bigcup_{\mu < \mu_{\ell}}\mathbb{M}_{h_{\ell}, \mu}$ for all $1\leq i\leq d+1$.
\end{Lem}\label{lem_the_snake}
\begin{proof} Assume that $\ell$ is the blue path in Figure \ref{fig:strip_to_resolve_ell}.  
We consider the minimal $(d+2)$-exact sequence given by the $(\ell_1,\ell_{d+1})$-strip in Figure \ref{fig:strip_to_resolve_ell}. 

\begin{equation*}
        \xymatrix@R=1em@C=1em{0\ar[r] & M_{\ell_1}\ar[r] & \cdots\ar[r]  & M_{\ell_{j}}\ar[r] & M_{\ell}\ar[r] & M_{\ell_{j+1}}\ar[r] & \cdots\ar[r] & M_{\ell_{d+1}}\ar[r] & 0}.
    \end{equation*}

\begin{figure}[h]
    \begin{center}
    \begin{tikzpicture}[scale=0.6]
    \coordinate (prev) at (0,0);
    \draw [color=green, line width=1] (-1,0)--(-1,3)--(1,3)--(1,4)--(2,4)--(2,7)--(3,7);
    \draw [color=green, line width=1] (3,6)--(5,6)--(5,8)--(6,8);
    \draw [color=blue, line width=1] (0,0)--(0,2)--(2,2)--(2,3)--(3,3)--(3,7)--(4,7)--(4,9)--(5,9);
    \draw [color=red, line width=1] (-1.11,0)--(0,2)--(1,2)--(3,5.6)--(3.77, 7)--(4.5, 8.3)--(4.88, 9);
    \draw (0,2) node [scale=0.3, circle, draw,fill=red]{};
    \draw (3,7) node [scale=0.3, circle, draw,fill=red]{};
    \draw (-1,3) node [scale=0.3, circle, draw,fill=red]{};
    \draw (2,7) node [scale=0.3, circle, draw,fill=red]{};
    \node[] at (0, 2.4) {$D$};
    \node[] at (3.3, 5.6) {$E$};
    \node[] at (2.7, 7.2) {$F$};
    \node[] at (4, 6.7) {$H$};
    \node[] at (2.6, 6.5) {$w$};
    \node[] at (3.5, 7.2) {$v$};
    \node[] at (-1, 3.4) {$D'$};
    \node[] at (1.7, 7.2) {$F'$};
    \node[] at (2.7, 6) {$E'$};
    \end{tikzpicture}
    \end{center}
    \caption{$\ell$ and the $(\ell_1, \ell_{d+1})$-strip}
    \label{fig:strip_to_resolve_ell}
\end{figure}
    
    Assume $1\leq i\leq j$. 
    Then the anchor of $\ell_i$ is $D$ and $\ell_i$ passes through $E'$ but not $F$.
    We claim $\mu_{\ell_i} <\mu_{\ell}$.
    Indeed, the contribution of $w_F$ decreases to either $w_{E'}$ or $0$ if $E'$ lies below $\delta_D$.
    Moreover, additional $w$-values decrease or remain the same.
    So $\ell_i\in \bigcup_{\mu<\mu_{\ell}}\mathbb{M}_{h_{\ell},\mu}$.
    
    Assume $j+1\leq i\leq d+1$. 
    If $\ell_i$ contains the segment from $D'$ to $F'$, then $D'$ is the anchor of $\ell_i$ and $h_{\ell_i} = h_{\ell}+1$.
    Otherwise, at least $\ell_i$ coincides with $\ell$ until $D$.
    None of the points lying on the segment before $D$ or $D$ itself are possible as anchors of $\ell_i$ since the line through $F'$ is further to the left. 
    Indeed this follows as the horizontal distance from $F'$ to $\delta_D$ is $1+v>1$.
    So $\ell_i\in \bigcup_{h>h_{\ell}}\mathbb{M}_h$.
\end{proof}

\begin{Prop}\label{prop_tilt_generating}
    We have $\thick\langle T \rangle = D^b(A)$.
\end{Prop}
\begin{proof}
    Firstly we claim that given an exact complex
    \begin{equation*}
        X_{\bullet}: 
        \xymatrix@R=1em@C=1em{0\ar[r] & X_1\ar[r] & \cdots\ar[r] & X_{j-1}\ar[r] & X_j\ar[r] & X_{j+1}\ar[r] & X_{a+1}\ar[r] & 0} \in \mathcal{D}^b(A)
    \end{equation*}
    with $X_i\in \thick\langle T\rangle$ for all $i\neq j$, then $X_j\in \thick\langle T\rangle$.

    To see this, denote by $\tau_{\leq i} X_{\bullet}$ the brutal truncation at $i$. We have the canonical triangle 
    \begin{equation*}
        \xymatrix@R=1em@C=1em{\tau_{\leq j}X_{\bullet}\ar[r] & \tau_{>j}X_{\bullet}\ar[r] & X_{\bullet}\ar[r] & \tau_{\leq j}X_{\bullet}[1]}.
    \end{equation*}
    We have $\tau_{\leq j}X_{\bullet}\cong \tau_{>j}X_{\bullet}\in \thick \langle T \rangle$ since $X_{\bullet}\cong 0$.

    Making use of the following triangle 
    \begin{equation*}
        \xymatrix@R=1em@C=1em{\tau_{< j}X_{\bullet}\ar[r] & X_j\ar[r] & \tau_{\leq j}X_{\bullet}\ar[r] & \tau_{< j}X_{\bullet}[1]},
    \end{equation*}
    we conclude that $X_j\in \thick \langle T \rangle$.
    
    Now we do double induction on $n-h$ and $\mu$ to show that $M_{\ell}\in \thick\langle T \rangle$ for all $\ell\in \mathbb{M}_{h\geq 1}$.

    The base cases are clear as $\mathbb{M}_{h,0}\subseteq \mathbb{T}$ for $h\geq 1$ and $\mathbb{M}_n=\mathbb{M}_{n,0}\subseteq \mathbb{T}$.

    Assume by induction hypothesis that $ M_{\ell}\in \thick\langle T\rangle$ for all  $\ell\in \bigcup_{h'>h}\mathbb{M}_{h'}\cup\bigcup_{\mu'<\mu}\mathbb{M}_{h,\mu'}$. Take $\ell\in \mathbb{M}_{h,\mu}$.
    We use the exact sequence in Lemma \ref{lem_the_resolving_exact_seq} to conclude that $M_{\ell}\in \thick\langle T\rangle$.

    Next we show $\ell\in \mathbb{M}_{\geq 1}$ for $M_{\ell}\in \add DA$. 
    Indeed, by Proposition \ref{prop_M_path} (i), $\ell$ passes through $(d,n)$. 
    So the anchor of $\ell$ is either $(d,n)$ or further to the left, which implies $h_{\ell}\geq 1$. 

    Therefore, $DA\in \thick\langle T\rangle$ and $\thick\langle T\rangle = \mathcal{D}^b(A)$. 
\end{proof}

\textbf{Proof of Theorem 4.5.} Proposition \ref{prop_tilt_rigidity} together with Proposition \ref{prop_tilt_generating} yields that $T$ is a tilting complex. Applying Theorem \ref{thm_construction_minimal_dRF_alg}, we have  $(i)$ and $(ii)$. Since fractionally Calabi-Yau property is a derived invariant, $(iii)$ follows from Proposition \ref{prop_A_n_d} $(v)$. 

\subsection{Endomorphism algebras}

Let $B_0 = \End_{\mathcal{D}^b(A)}(P)$. 
By applying certain constructions to $B_0$, we obtain $B$ as well as its $nd$-Auslander algebra and $(nd+1)$-preprojective algebra. 

Let $r\geq 1$ be an integer.
Recall the $r$-replicated algebra (see \cite{AI87})  of an algebra $A$ is given by the following $r\times r$ matrix algebra.

\begin{equation*}
    A^{(r)} = 
    \begin{pmatrix}
        A & DA & 0 & \cdots & 0 \\
        0 & A & DA & \cdots & 0\\
        \vdots & \vdots & \ddots & \ddots & \vdots\\
        0 & 0 & \cdots & A & DA\\
        0 & 0 & \cdots & 0 & A
    \end{pmatrix}.
\end{equation*}

\begin{Prop}\label{prop_B_B0}
    We have $B\cong B_0^{(n+d)}$.
\end{Prop}
\begin{proof}
    Recall the general form of $B$ from Section \ref{section_construction}.
In our special case, 
\begin{equation*}
    B_i = \Hom_{\mathcal{D}^b(A)}(P, \nu^iP) = 0
\end{equation*}
for $i\geq 2$ by Colollary \ref{cor_hom_p_to_nu_power_p}. 
Therefore, $B\cong B_0^{(n+d)}$.
\end{proof}

\begin{Eg}
    A concrete example of Proposition \ref{prop_B_B0} is given in Example \ref{eg_A_5_3_B}. Here $B_0\cong kQ/\langle \gamma\alpha, \delta\beta-\mu\gamma \rangle$ in which $Q$ is given below.
    
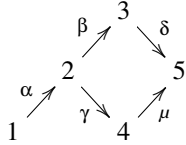
\begin{equation*}
        \xymatrix@C=1em@R=1em{
        && 3\ar[dr]^{\delta}\\
        & 2\ar[ur]^{\beta}\ar[dr]_{\gamma} && 5\\
        1\ar[ur]^{\alpha} && 4\ar[ur]_{\mu}
        }
    \end{equation*}
\end{Eg}

Recall in Section \ref{section_construction} that
\begin{equation*}
    N = F\left(\bigoplus_{i=1}^{n+d+1}\nu^iP\right)
\end{equation*}
is the basic $nd$-cluster tilting module of $B$.

\begin{Prop}\label{prop_hig_aus_alg_of_B}
    Denote by $\Lambda=\End_B(N)$ the $nd$-Auslander algebra of $B$. Then $\Lambda \cong B_0^{(n+d+1)}$.
    In particular, $\gldim B_0^{(n+d+1)}\leq nd+1 \leq \domdim B_0^{(n+d+1)}$.
\end{Prop}
\begin{proof}
    Combining the general form in Section \ref{section_construction} of $\Lambda$ and Corollary \ref{cor_hom_p_to_nu_power_p}, 
    the statement follows.
    The second part follows from Proposition \ref{prop_aus_alg_gldim_domdim}.
\end{proof}

\begin{Eg}
    Let $A=A_5^3$. See Example \ref{eg_A_5_3_B} for the quiver of $B$. The $12$-Auslander algebra $\Lambda$ of $B$ is the path algebra of the following quiver with relations, where the quiver is given below and the relation is generated by zero relations (red dashed lines) and commutative relations (blue dashed lines). Consequently $\gldim \Lambda = 13 = \domdim \Lambda$ by Proposition \ref{prop_aus_alg_gldim_domdim}.
    \begin{equation*}
        \begin{xy}
            0;<18pt,0cm>:<8pt,24pt>::
            (0,0) *+{1} ="1",
            (0,1) *+{2} ="2",
            (0,2) *+{3} ="3",
            (0,3) *+{6} ="6",
            (1.5,0) *+{4} ="4",
            (1.5,1) *+{5} ="5",
            (1.5,2) *+{7} ="7",
            (1.5,3) *+{9} ="9",
            (3,1) *+{8} ="8",
            (3,2) *+{10} ="10",
            (4.5,0) *+{11} ="11",
            (4.5,1) *+{12} ="12",
            (4.5,2) *+{13} ="13",
            (4.5,3) *+{16} ="16",
            (6,0) *+{14} ="14",
            (6,1) *+{15} ="15",
            (6,2) *+{17} ="17",
            (6,3) *+{19} ="19",
            (7.5,1) *+{18} ="18",
            (7.5,2) *+{20} ="20",
            (9,0) *+{21} ="21",
            (9,1) *+{22} ="22",
            (9,2) *+{23} ="23",
            (9,3) *+{26} ="26",
            (10.5,0) *+{24} ="24",
            (10.5,1) *+{25} ="25",
            (10.5,2) *+{27} ="27",
            (10.5,3) *+{29} ="29",
            (12,1) *+{28} ="28",
            (12,2) *+{30} ="30",
            (13.5,0) *+{31} ="31",
            (13.5,1) *+{32} ="32",
            (13.5,2) *+{33} ="33",
            (13.5,3) *+{36} ="36",
            (15,0) *+{34} ="34",
            (15,1) *+{35} ="35",
            (15,2) *+{37} ="37",
            (15,3) *+{39} ="39",
            (16.5,1) *+{38} ="38",
            (16.5,2) *+{40} ="40",
            "1", {\ar "2"},
            "2", {\ar "3"},
            "3", {\ar "6"},
            "2", {\ar "4"},
            "3", {\ar "5"},
            "6", {\ar "7"},
            "4", {\ar "5"},
            "5", {\ar "7"},
            "7", {\ar "9"},
            "7", {\ar "8"},
            "9", {\ar "10"},
            "8", {\ar "10"},
            "8", {\ar "11"},
            "10", {\ar "12"},
            "11", {\ar "12"},
            "12", {\ar "13"},
            "13", {\ar "16"},
            "12", {\ar "14"},
            "13", {\ar "15"},
            "16", {\ar "17"},
            "14", {\ar "15"},
            "15", {\ar "17"},
            "17", {\ar "19"},
            "17", {\ar "18"},
            "18", {\ar "20"},
            "19", {\ar "20"},
            "18", {\ar "21"},
            "20", {\ar "22"},
            "21", {\ar "22"},
            "22", {\ar "23"},
            "23", {\ar "26"},
            "22", {\ar "24"},
            "23", {\ar "25"},
            "26", {\ar "27"},
            "24", {\ar "25"},
            "25", {\ar "27"},
            "27", {\ar "29"},
            "27", {\ar "28"},
            "29", {\ar "30"},
            "28", {\ar "30"},
            "28", {\ar "31"},
            "30", {\ar "32"},
            "31", {\ar "32"},
            "32", {\ar "33"},
            "32", {\ar "34"},
            "33", {\ar "35"},
            "34", {\ar "35"},
            "33", {\ar "36"},
            "35", {\ar "37"},
            "37", {\ar "39"},
            "36", {\ar "37"},
            "37", {\ar "38"},
            "39", {\ar "40"},
            "38", {\ar "40"},
            "1", {\ar@[red]@{--} "4"},
            "6", {\ar@[red]@{--} "9"},
            "11", {\ar@[red]@{--} "14"},
            "16", {\ar@[red]@{--} "19"},
            "21", {\ar@[red]@{--} "24"},
            "26", {\ar@[red]@{--} "29"},
            "31", {\ar@[red]@{--} "34"},
            "36", {\ar@[red]@{--} "39"},
            "4", {\ar@[red]@{--} "8"},
            "5", {\ar@[red]@{--} "11"},
            "9", {\ar@[red]@{--} "13"},
            "10", {\ar@[red]@{--} "16"},
            "14", {\ar@[red]@{--} "18"},
            "15", {\ar@[red]@{--} "21"},
            "19", {\ar@[red]@{--} "23"},
            "20", {\ar@[red]@{--} "26"},
            "24", {\ar@[red]@{--} "28"},
            "25", {\ar@[red]@{--} "31"},
            "29", {\ar@[red]@{--} "33"},
            "30", {\ar@[red]@{--} "36"},
            "34", {\ar@[red]@{--} "38"},
            "2", {\ar@[blue]@{--} "5"},
            "3", {\ar@[blue]@{--} "7"},
            "7", {\ar@[blue]@{--} "10"},
            "8", {\ar@[blue]@{--} "12"},
            "12", {\ar@[blue]@{--} "15"},
            "13", {\ar@[blue]@{--} "17"},
            "17", {\ar@[blue]@{--} "20"},
            "18", {\ar@[blue]@{--} "22"},
            "22", {\ar@[blue]@{--} "25"},
            "23", {\ar@[blue]@{--} "27"},
            "27", {\ar@[blue]@{--} "30"},
            "28", {\ar@[blue]@{--} "32"},
            "32", {\ar@[blue]@{--} "35"},
            "33", {\ar@[blue]@{--} "37"},
            "37", {\ar@[blue]@{--} "40"},
        \end{xy}
    \end{equation*}
\end{Eg}

Next we show that the $(nd+1)$-preprojective algebra $\Pi_{nd+1}(B)$ of $B$ is given by the $(n+d)$-fold trivial extension algebra of $B_0$. 
Recall that the $r$-fold trivial extension algebra $T_r(A)$ is given by the following $r\times r$ matrix algebra, see \cite{CDIM25} for more details.

\begin{equation*}
    T_r(A) = 
    \begin{pmatrix}
        A & DA & 0 & \cdots & 0 \\
        0 & A & DA & \cdots & 0\\
        \vdots & \vdots & \ddots & \ddots & \vdots\\
        0 & 0 & \cdots & A & DA\\
        DA & 0 & \cdots & 0 & A
    \end{pmatrix}
\end{equation*}
When $r=1$, $T_1(A)$ is the trivial extension of $A$ and we follow the classical notation and denote it by $T(A)$.

We consider $T_r(A)$ with a particular grading $T_r(A) = T_r(A)_0\oplus T_r(A)_1$ where 
\begin{equation*}
    T_r(A)_0 = 
    \begin{pmatrix}
        A & DA & 0 & \cdots & 0 \\
        0 & A & DA & \cdots & 0\\
        \vdots & \vdots & \ddots & \ddots & \vdots\\
        0 & 0 & \cdots & A & DA\\
        0 & 0 & \cdots & 0 & A
    \end{pmatrix} \text{ and }
    T_r(A)_1 = 
    \begin{pmatrix}
        0 & 0 & 0 & \cdots & 0 \\
        0 & 0 & 0 & \cdots & 0\\
        \vdots & \vdots & \ddots & \ddots & \vdots\\
        0 & 0 & \cdots & 0 & 0\\
        DA & 0 & \cdots & 0 & 0
    \end{pmatrix}.
\end{equation*}

\begin{Prop}\label{prop_hig_preproj_alg_of_B}
    We have $\Pi_{nd+1}(B) \cong T_{n+d}(B_0)$ as graded algebras.
\end{Prop}
\begin{proof}
    By Proposition \ref{prop_B_B0}, we have $\Pi_{nd+1}(B)_0 = B \cong T_{n+d}(B_0)_0$.
    Next we calculate the positive degree part of $\Pi_{nd+1}(B)$. 
    \begin{align*}
        \Pi_{nd+1}(B)_{\geq 1} & \cong \bigoplus_{i\geq 1}\Hom_{\mathcal{D}^b(B)}(B, \nu_{B,nd}^{-i}(B))\\
        & \cong \bigoplus_{i\geq 1} \Hom_{\mathcal{D}^b(A)}\left(\bigoplus_{j=1}^{n+d}\nu_A^jP, \nu_{A, nd}^{-i}\bigoplus_{k=1}^{n+d}\nu_A^kP\right)\\
        & \cong \bigoplus_{i\geq 1} \Hom_{\mathcal{D}^b(A)}\left(\bigoplus_{j=1}^{n+d}\nu_A^jP, \bigoplus_{k=1}^{n+d}\nu_A^{i(n+d)+k}P\right) \text{ by Proposition \ref{prop_A_n_d} } (iv)\\
        & \cong \bigoplus_{i\geq 1} \bigoplus_{1\leq j,k\leq n+d} \Hom_{\mathcal{D}^b(A)}(P, \nu_A^{i(n+d)+k-j}P).
    \end{align*}
    By Corollary \ref{cor_hom_p_to_nu_power_p}, 
    $\Hom_{\mathcal{D}^b(A)}(P, \nu_A^{i(n+d)+k-j}P)\neq 0$ only if $i=1$, $k=1$ and $j=n+d$.
    And in this case by Serre duality, $\Hom_{\mathcal{D}^b(A)}(P, \nu_A P)\cong DB_0$.
    Thus we have $\Pi_{nd+1}(B)_1 \cong T_{n+d}(B_0)_1$.
\end{proof}

\begin{Eg}
    Let $A=A_5^3$. See Example \ref{eg_A_5_3_B} for the quiver of $B$. The $13$-preprojective algebra of $B$ is the path algebra of the following quiver with relations, where the quiver is given below and the relations are generated by zero relations (red dashed lines) and commutative relations (blue dashed lines). 
    Note that the $5$ vertices on the bottom left corner and the $5$ vertices on the up right corner are identified.
    
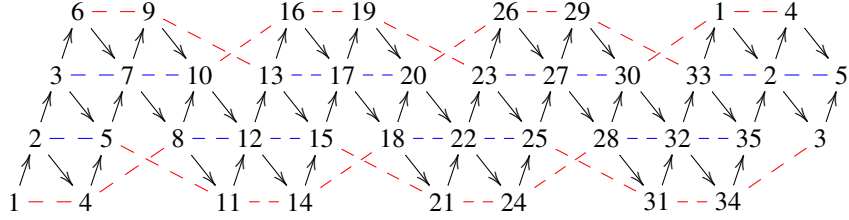
\begin{equation*}
        \begin{xy}
            0;<18pt,0cm>:<8pt,24pt>::
            (0,0) *+{1} ="1",
            (0,1) *+{2} ="2",
            (0,2) *+{3} ="3",
            (0,3) *+{6} ="6",
            (1.5,0) *+{4} ="4",
            (1.5,1) *+{5} ="5",
            (1.5,2) *+{7} ="7",
            (1.5,3) *+{9} ="9",
            (3,1) *+{8} ="8",
            (3,2) *+{10} ="10",
            (4.5,0) *+{11} ="11",
            (4.5,1) *+{12} ="12",
            (4.5,2) *+{13} ="13",
            (4.5,3) *+{16} ="16",
            (6,0) *+{14} ="14",
            (6,1) *+{15} ="15",
            (6,2) *+{17} ="17",
            (6,3) *+{19} ="19",
            (7.5,1) *+{18} ="18",
            (7.5,2) *+{20} ="20",
            (9,0) *+{21} ="21",
            (9,1) *+{22} ="22",
            (9,2) *+{23} ="23",
            (9,3) *+{26} ="26",
            (10.5,0) *+{24} ="24",
            (10.5,1) *+{25} ="25",
            (10.5,2) *+{27} ="27",
            (10.5,3) *+{29} ="29",
            (12,1) *+{28} ="28",
            (12,2) *+{30} ="30",
            (13.5,0) *+{31} ="31",
            (13.5,1) *+{32} ="32",
            (13.5,2) *+{33} ="33",
            (13.5,3) *+{1} ="36",
            (15,0) *+{34} ="34",
            (15,1) *+{35} ="35",
            (15,2) *+{2} ="37",
            (15,3) *+{4} ="39",
            (16.5,1) *+{3} ="38",
            (16.5,2) *+{5} ="40",
            "1", {\ar "2"},
            "2", {\ar "3"},
            "3", {\ar "6"},
            "2", {\ar "4"},
            "3", {\ar "5"},
            "6", {\ar "7"},
            "4", {\ar "5"},
            "5", {\ar "7"},
            "7", {\ar "9"},
            "7", {\ar "8"},
            "9", {\ar "10"},
            "8", {\ar "10"},
            "8", {\ar "11"},
            "10", {\ar "12"},
            "11", {\ar "12"},
            "12", {\ar "13"},
            "13", {\ar "16"},
            "12", {\ar "14"},
            "13", {\ar "15"},
            "16", {\ar "17"},
            "14", {\ar "15"},
            "15", {\ar "17"},
            "17", {\ar "19"},
            "17", {\ar "18"},
            "18", {\ar "20"},
            "19", {\ar "20"},
            "18", {\ar "21"},
            "20", {\ar "22"},
            "21", {\ar "22"},
            "22", {\ar "23"},
            "23", {\ar "26"},
            "22", {\ar "24"},
            "23", {\ar "25"},
            "26", {\ar "27"},
            "24", {\ar "25"},
            "25", {\ar "27"},
            "27", {\ar "29"},
            "27", {\ar "28"},
            "29", {\ar "30"},
            "28", {\ar "30"},
            "28", {\ar "31"},
            "30", {\ar "32"},
            "31", {\ar "32"},
            "32", {\ar "33"},
            "32", {\ar "34"},
            "33", {\ar "35"},
            "34", {\ar "35"},
            "33", {\ar "36"},
            "35", {\ar "37"},
            "37", {\ar "39"},
            "36", {\ar "37"},
            "37", {\ar "38"},
            "39", {\ar "40"},
            "38", {\ar "40"},
            "1", {\ar@[red]@{--} "4"},
            "6", {\ar@[red]@{--} "9"},
            "11", {\ar@[red]@{--} "14"},
            "16", {\ar@[red]@{--} "19"},
            "21", {\ar@[red]@{--} "24"},
            "26", {\ar@[red]@{--} "29"},
            "31", {\ar@[red]@{--} "34"},
            "36", {\ar@[red]@{--} "39"},
            "4", {\ar@[red]@{--} "8"},
            "5", {\ar@[red]@{--} "11"},
            "9", {\ar@[red]@{--} "13"},
            "10", {\ar@[red]@{--} "16"},
            "14", {\ar@[red]@{--} "18"},
            "15", {\ar@[red]@{--} "21"},
            "19", {\ar@[red]@{--} "23"},
            "20", {\ar@[red]@{--} "26"},
            "24", {\ar@[red]@{--} "28"},
            "25", {\ar@[red]@{--} "31"},
            "29", {\ar@[red]@{--} "33"},
            "30", {\ar@[red]@{--} "36"},
            "34", {\ar@[red]@{--} "38"},
            "2", {\ar@[blue]@{--} "5"},
            "3", {\ar@[blue]@{--} "7"},
            "7", {\ar@[blue]@{--} "10"},
            "8", {\ar@[blue]@{--} "12"},
            "12", {\ar@[blue]@{--} "15"},
            "13", {\ar@[blue]@{--} "17"},
            "17", {\ar@[blue]@{--} "20"},
            "18", {\ar@[blue]@{--} "22"},
            "22", {\ar@[blue]@{--} "25"},
            "23", {\ar@[blue]@{--} "27"},
            "27", {\ar@[blue]@{--} "30"},
            "28", {\ar@[blue]@{--} "32"},
            "32", {\ar@[blue]@{--} "35"},
            "33", {\ar@[blue]@{--} "37"},
            "37", {\ar@[blue]@{--} "40"},
        \end{xy}
    \end{equation*}
\end{Eg}

Recall that for an algebra $A$ with $\gldim A<\infty$, Happel gave a triangle equivalence \cite{Hap88}
\begin{equation*}
    \mathcal{D}^b(A) \cong \underline{\modn}^{\mathbb{Z}} T(A).
\end{equation*}
The uniqueness of Serre functor implies that the following diagram commutes up to isomorphism of functors.
\begin{equation*}
    \xymatrix{
        \mathcal{D}^b(A)\ar[r]^-{\sim}\ar[d]^{\nu_{A}} &  \underline{\modn}^{\mathbb{Z}}T(A)\ar[d]^{\Omega\circ (1)}\\
        \mathcal{D}^b(A)\ar[r]^-{\sim} & \underline{\modn}^{\mathbb{Z}}T(A).
        }
\end{equation*}

For an algebra $A$, denote by $A^e = A\otimes_k A^{op}$ the enveloping algebra. 
Recall that an algebra $A$ is called twisted periodic if $\Omega_{A^e}^n\cong _1A_{\phi}$ in $\modn A^e$ for some integer $n\geq 1$ and $\phi:A\rightarrow A$ is an algebra automorphism. In particular, $A$ is called periodic if $\phi$ is the identity morphism. 
We recall the following result from \cite{CDIM25}.

\begin{Prop}(\cite[Theorem 1.4, Proposition 4.4]{CDIM25})\label{prop_CDIM_results}
    Let $A$ be a finite dimensional algebra over a field $k$ such that $A/\rad A$ is a separable $k$-algebra. The following conditions are equivalent.
    \begin{itemize}
        \item [(i)] $T(A)$ is twisted periodic.
        \item [(ii)] There exist $d,r\geq 1$ such that $T_r(A)$ admits a $d$-cluster tilting module.
        \item [(iii)] $A$ has finite global dimension and is twisted fractionally Calabi-Yau.
    \end{itemize}
    Assume that $\gldim A < \infty$. Given $m,  \ell \in \mathbb{Z}$ with $\ell > 0$,  the following conditions are equivalent. 
    \begin{itemize}
        \item [(a)] The algebra $A$ is $\frac{m}{\ell}$-Calabi-Yau.
        \item [(b)] There is an isomorphism of functors $\Omega^{\ell+m}\cong (-\ell)$ on $\underline{\modn} T(A)$.
    \end{itemize}
\end{Prop}

We obtain the following Corollary by applying Proposition \ref{prop_CDIM_results} to related algebras constructed from $B_0$.

\begin{Cor}\label{cor_B0_repetative_r_fold_trivial_ext}
    Assume $\gcd(n,d) = 1$. Then the following statements hold.
    \begin{itemize}
        \item [(i)] $B_0$ has finite global dimension and is twisted fractionally Calabi-Yau.
        \item [(ii)] $B_0^{(n+d)}$ is $nd$-representation finite and $\frac{nd}{n+d+1}$-Calabi-Yau.
        \item [(iii)] $T_{n+d}(B_0)$ is self-injective. Moreover, $\underline{\modn} T_{n+d}(B_0)$ is $(nd+1)$-Calabi-Yau and admits an $(nd+1)$-cluster tilting module.
    \end{itemize}
\end{Cor}
\begin{proof}
    Proposition \ref{prop_hig_preproj_alg_of_B} together with Proposition \ref{prop_properties_hig_preproj_alg} implies (iii).
    By Proposition \ref{prop_CDIM_results}, (iii) implies (i).
    Lastly, (ii) follows from Proposition \ref{prop_B_B0} and Theorem \ref{thm_main_construction_higher_auslander_alg_typeA}.
\end{proof}

Denote by $s=\ulcorner \frac{d}{n} \urcorner$.
Now we describe $B_0$ as an idempotent subalgebra of an $(d-s-1)$-Auslander algebra of type $\mathbb{A}_{n+1}$. This will allow us to make Corollary \ref{cor_B0_repetative_r_fold_trivial_ext} (i) more precise.
Define the map 
    \begin{align*}
        \alpha: \os_{n+1}^d & \rightarrow \os_{n+1}^{d-s}\\
        (x_1, \ldots, x_{d}) & \mapsto (x_{s+1}-s, \ldots, x_{d}-s).
    \end{align*} 
We denote by $\beta=\alpha c: \Dyck_{d,n}\rightarrow \os_{n+1}^{d-s}$.
We consider the algebra $A'= A_{n+1}^{d-s}$, in which the vertex set of its quiver is given by $\os_{n+1}^{d-s}$. 
Let $e\in A'$ be the idempotent defined as $e = \sum_{\ell\in \Dyck_{n,d}}e_{\beta(\ell)}$. 
The following statement holds.

\begin{Prop} We have 
    $B_0\cong eA'e \cong A'/\langle 1-e\rangle$.
\end{Prop}
\begin{proof}
    Denote by $Q_{B_0}$ (respectively $Q_{A'}$) the quiver of $B_0$ (respectively $A'$).
    Let $\beta(u\rightarrow v) = (\beta(u)\rightarrow \beta(v))$. 
    Then $\beta$ identifies $Q_{B_0}$ as a subquiver of $Q_{A'}$. It can be checked that $\beta$ preserves all  relations. Thus $B_0\cong eA'e$. 

    We need to show $eA'e\cong A'/\langle 1-e \rangle$. Note that $x<y$ for $x\in \beta(\Dyck_{d,n})$ and $y\in \os_{n+1}^{d-s}\backslash \beta(\Dyck_{d,n})$. This implies that there are no arrows from $\os_{n+1}^{d-s}\backslash \beta(\Dyck_{d,n})$ to $\beta(\Dyck_{d,n})$. 
    In other words, we have $eA'(1-e) = 0$ and
    \begin{equation*}
        A'\cong 
        \begin{pmatrix}
            eA'e & 0\\
            (1-e)A'e & (1-e)A'(1-e)
        \end{pmatrix}.
    \end{equation*}
    Therefore, $eA'e \cong A'/\langle 1-e \rangle$.
\end{proof}

\begin{Prop}\label{prop_gldim_B0_fCY_B0}
    Assume $\gcd(n,d)=1$. Denote by $s = \ulcorner \frac{d}{n} \urcorner$. We have 
    \begin{itemize}
        \item [(i)] $\gldim B_0 = d-s$, and
        \item [(ii)] $B_0$ is $\frac{(n-1)(d-1)}{n+d+1}$-Calabi-Yau. 
    \end{itemize}
\end{Prop}
\begin{proof}   
    Consider the adjoint pair $(F, G)$ defined as follows. 
    \begin{equation*}
        \xymatrix@R=6em{
        \modn A' \ar@/_/[rr]_{G=\Hom_{A'}(eA', -)} && \modn eA'e \cong \modn B_0\ar@/_/[ll]_{F = -\otimes_{eA'e} eA'}
        }
    \end{equation*}
    Since $eA' \cong eA'e\oplus eA'(1-e) = eA'e$ as $eA'e$-modules, $F$ is exact. 
    Moreover, $F$ preserves projective modules since it admits a right adjoint $G$ which is exact. 
    We may view $F$ as an exact embedding and conclude that $\projdim FN = \projdim N$ for $N\in \modn B_0$. 
    Thus $\gldim B_0\leq \gldim A' = d-s$.
    
    We need to show that $\gldim B_0 = d-s$. For this, let $\ell_m\in \Dyck_{d,n}$ be the maximum element. We have $c(\ell_m) = (1,2,\ldots, s, m_{s+1}, \ldots, m_d)$ with $m_{s+1}>s+1$. 
    
    Denote by $S_{\ell_m}$ the simple $B_0$-module at vertex $\ell_m$. We claim that $\projdim S_{\ell_m} = d-s$ which will imply $\gldim B_0 = d-s$. 
    To see this, we show instead $\projdim FS_{\ell_m} = d-s$.
     
    Recall that $A' \cong \End_{A_{n+1}^{d-s-1}}(M_{n+1}^{d-s-1})$. Denote by $P_{\ell_m}\in \add A'$ the projective cover of $FS_{\ell_m}$. Since $\beta(\ell_m) = (m_{s+1}-s, \ldots, m_d-s)$ and $m_{s+1}-s>1$, we have that 
    \begin{equation*}
        P_{\ell_m} \cong \Hom_{A_{n+1}^{d-s-1}}(M_{n+1}^{d-s-1}, M)
    \end{equation*}
    with $M\in \add M_{n+1}^{d-s-1} \backslash \add A_{n+1}^{d-s-1}$. 
    Thus there exists the following $(d-s+1)$-exact sequence in $\add M_{n+1}^{d-s-1}$. 
    \begin{equation*}
        \xymatrix@R=1em@C=1em{0\ar[r] & \tau_{d-s-1}M\ar[r] & N_{d-s-1}\ar[r] & \cdots\ar[r] & N_1\ar[r] & M\ar[r] & 0}.
    \end{equation*}
    Applying $H:=\Hom_{A_{n+1}^{d-s-1}}(M_{n+1}^{d-s-1}, -)$ to it gives the minimal projective resolution of $FS_{\ell_m}$.
    \begin{equation*}
        \xymatrix@R=1em@C=1em{0\ar[r] & H(\tau_{d-s-1}M)\ar[r] & HN_{d-s-1}\ar[r] & \cdots\ar[r] & HN_1\ar[r] & P_{\ell_m}\ar[r] & FS_{\ell_m}\ar[r] & 0}.
    \end{equation*}
    We have that $\projdim FS_{\ell_m} = d-s$ as claimed. So (i) follows.

    We will apply Proposition \ref{prop_CDIM_results} to show (ii). So instead we calculate the periodicity of $T(B_0)$. Denote by $(1)$ the degree shift of $T(B_0)$. We have the following commutative diagram.
    \begin{equation*}
        \xymatrix{
        \underline{\modn}^{\mathbb{Z}}T(B_0)\ar[r]^{\sim}\ar[d]^{(-(n+d))} & \underline{\modn}^{\mathbb{Z}}T_{n+d}(B_0)\ar[d]^{(-1)}\\
        \underline{\modn}^{\mathbb{Z}}T(B_0)\ar[r]^{\sim} & \underline{\modn}^{\mathbb{Z}}T_{n+d}(B_0).
        }
    \end{equation*}
    By Proposition \ref{prop_hig_preproj_alg_of_B}, we have $\underline{\modn}^{\mathbb{Z}}T_{n+d}(B_0)\cong \underline{\modn}^{\mathbb{Z}}\Pi_{nd+1}(B)$.
    On the other hand, since 
    \begin{equation*}
        \mathcal{U} = \mathcal{U}_{nd}(B) = \add \{\Pi_{nd+1}(B)[ind]\mid i\in \mathbb{Z}\},
    \end{equation*}
    we have that $\proj^{\mathbb{Z}}\Pi_{nd+1}(B) \cong \mathcal{U}$. 
    Thus $\underline{\modn}^{\mathbb{Z}}\Pi_{nd+1}(B)\cong \underline{\modn}\mathcal{U}$.
    Denote by $\mathbb{S}_{\mathcal{U}}$ (respectively $[1]_{\mathcal{U}}$ ) the Serre functor (respectively the shift functor) on $\underline{\modn} \mathcal{U}$. 
    The following diagram commutes.
    \begin{equation*}
        \xymatrix{
        \underline{\modn}^{\mathbb{Z}}\Pi_{nd+1}(B)\ar[r]^-{\sim}\ar[d]^{(-1)} &  \underline{\modn}\mathcal{U}\ar[d]^{\nu_{B,nd}^{\ast}}\\
        \underline{\modn}^{\mathbb{Z}}\Pi_{nd+1}(B)\ar[r]^-{\sim} &  \underline{\modn}\mathcal{U}.
        }
    \end{equation*}
    By \cite[Theorem 4.5]{IO13},  
    \begin{equation*}
        \mathbb{S}_{\mathcal{U}}\circ [-nd-1]_{\mathcal{U}} \cong \nu_{B,nd}^{\ast}.
    \end{equation*}
    Thus 
    \begin{equation*}
        \xymatrix{
        \underline{\modn}^{\mathbb{Z}}T(B_0)\ar[r]^{\sim}\ar[d]^{(-(n+d))} & \underline{\modn}\mathcal{U}\ar[d]^{\mathbb{S}_{\mathcal{U}}\circ [-nd-1]_{\mathcal{U}}}\\
        \underline{\modn}^{\mathbb{Z}}T(B_0)\ar[r]^{\sim} & \underline{\modn}\mathcal{U}.
        }
    \end{equation*}
    This implies 
    \begin{equation*}
        (-(n+d)) \cong \mathbb{S}_{T(B_0)}\circ \Omega^{nd+1},
    \end{equation*}
    where $\mathbb{S}_{T(B_0)}$ is the Serre functor on $\underline{\modn}^{\mathbb{Z}}T(B_0)$.
    
    Meanwhile $\mathbb{S}_{T(B_0)} \cong \Omega\circ (1)$. 
    Therefore, 
    \begin{equation*}
        (-(n+d)) \cong \Omega\circ (1)\circ \Omega^{nd+1},
    \end{equation*}
    which implies $\Omega^{nd+2} \cong (-(n+d+1))$.
    By Proposition \ref{prop_CDIM_results}, we have that $B_0$ is $\frac{(n-1)(d-1)}{n+d+1}$-Calabi-Yau.
\end{proof}

\bibliographystyle{amsplain}

\begin{thebibliography}{99}

   \bibitem[AI87]{AI87}{{\sc I.Assem} and {Y. Iwanaga}, On a class of representation finite QF-3 algebras, \emph{Tsukuba J.Math.} \textbf{11} (2) (1987) 199-217.}

   \bibitem[Biz54]{Biz54}{{\sc M.T.L. Bizley}, Derivation of a new formula for the number of minimal lattice paths from $(0, 0)$ to $(km, kn)$ having just $t$ contacts with the line $my = nx$ and having no points above this line; and a proof of Grossman’s formula for the number of paths which may touch but do not rise above this line, \emph{J.Inst.Actuar.} \textbf{80} (1954) 55-62.}

   \bibitem[CDIM25]{CDIM25}{{\sc A. Chan}, {\sc E. Darp\"{o}}, {\sc O. Iyama} and {\sc R. Marczinzik}, Periodic trivial extension algebras and fractionally Calabi-Yau algebras, \emph{Ann. Sci. Éc. Norm. Supér.} (4) 58 (2025), no. 2, 463--510.}


   \bibitem[Cha23]{Cha23}{{\sc F. Chapton}, Posets and fractionally Calabi-Yau categories, \emph{arXiv: 2303.11656.}}


   \bibitem[CLR18]{CLR18}{{\sc F. Chapoton}, {\sc S. Ladkani} and {\sc B. Rognerud}, On derived equivalences for categories of generalized intervals of a finite poset, \emph{arXiv:1801.05154}.}

   \bibitem[Ded23]{Ded23}{{\sc I.Dedda}, Symplectic higher Auslander correspondence for type $A$, \emph{	arXiv:2311.16859}.}

   \bibitem[DI20]{DI20}{{\sc E.Darpö} and {\sc O.Iyama}, $d$-representation-finite self-injective algebras, \emph{Adv.Math.} \textbf{362} (2020) 106932.}


   \bibitem[DJL21]{DJL21}{{\sc T. Dyckerhoff}, {\sc G. Jasso} and {Y. Lekili}, The symplectic geometry of higher Auslander algebras: Symmetric products of disks, \emph{Forum Math. Sigma} \textbf{9} (2021) 1-49.}

   \bibitem[DJW19]{DJW19}{{\sc T. dyckerhoff}, {\sc G. Jasso} and {\sc T. Walde}, Simplicial structures in higher Auslander-Reiten theory, \emph{Adv.Math.} \textbf{355} 106762 (2019).}
   
   \bibitem[GKO13]{GKO13}{{\sc C. Geiss}, {\sc B. Keller} and {\sc S. Oppermann}, {$n$}-angulated categories, \emph{J. Reine Angew. Math.} \textbf{675} (2013) 101-120.}

   \bibitem[Got24]{Got24}{{\sc T. Gottesman}, Fractionally Calabi-Yau lattices that tilt to higher Auslander algebras of type $A$, \emph{arXiv:2406.09148}.}

   \bibitem[Gra23]{Gra23}{{\sc J. Grant}, Serre functors and graded categories, \emph{Algebr. Represent. Theory} \textbf{26} (2023) 2113-2180.}
  
   \bibitem[Hap88]{Hap88}{{\sc D. Happel}, \emph{Triangulated categories in the representation theory of finite-dimensional algebras}, London Mathematical Society Lecture Note Series, vol. 119, Cambridge University Press, Cambridge, 1988.}

   \bibitem[HI11a]{HI11a}{{\sc M.Herschend} and {\sc O.Iyama}, $n$-representation finite algebras and twisted fractionally Calabi-Yau algebras, \emph{Bull.Lond.Math.Soc} \textbf{43} (3) (2011) 449-466.}

   \bibitem[HI11b]{HI11b}{{\sc M.Herschend} and {\sc O.Iyama}, Selfinjective quivers with potential and $2$-representation-finite algebras, \emph{Compos. Math.} \textbf{147} (6) (2011) 1885-1920.}

   \bibitem[HIMO23]{HIMO23}{{\sc M.Herschend}, {\sc O. Iyama}, {\sc H. Minomoto} and {\sc S. Oppermann}, Representation theory of the Geigle-Lenzing complete intersections, \emph{Mem. Am. Math. Soc.} \textbf{285} (1412) (2023) 8-155.}

   \bibitem[HIO14]{HIO14}{{\sc M. Herschend}, {\sc O. Iyama} and {\sc S. Oppermann}, $n$-representation infinite algebras, \emph{Adv.Math.} \textbf{252} (2014) 292-342.}
   
   \bibitem[HJ21]{HJ21}{{\sc M. Herschend} and {\sc P. J\o rgensen}, Classification of higher wide subcategories for higher {A}uslander algebras of type {$A$}, \emph{J. Pure Appl. Algebra} \textbf{225}(5): Paper No. 106583, 22, 2021. }
   
   \bibitem[HJV20]{HJV20}{{\sc M. Herschend}, {\sc P. J\o rgensen} and {\sc L. Vaso}, Wide subcategories of {$d$}-cluster tilting subcategories, \emph{Trans. Amer. Math. Soc.} \textbf{373}(4) (2020) 2281--2309. }
   
   \bibitem[HKV25]{HKV25}{{\sc M. Herschend}, {\sc S. Kvamme} and {\sc L. Vaso}, $n\mathbb{Z}$-cluster tilting subcategories for Nakayama algebras, \emph{Math.Z.} \textbf{309} 37 (2025).}

   \bibitem[IJ17]{IJ17}{{\sc O. Iyama} and {\sc G. Jasso}, Higher {A}uslander correspondence for dualizing {$R$}-varieties, \emph{Algebr. Represent. Theory} \textbf{20} (2017) 335-354.}

   \bibitem[IO11]{IO11}{{\sc O. Iyama} and {\sc S. Oppermann}, $n$-representation-finite algebras and $n$-APR tilting, \emph{Trans. Amer. Math. Soc.} \textbf{363} (12) (2011) 6575-6614.}
   
   \bibitem[IO13]{IO13}{{\sc O. Iyama} and {\sc S. Oppermann}, Stable categories of higher preprojective algebras, \emph{Adv. Math.} \textbf{244} (2013) 23-68.}

   \bibitem[IW11]{IW11}{{\sc O. Iyama} and {M. Weymss}, A new triangulated category for rational surface singularities, \emph{Illinois J. Math.} \textbf{55} (1) (2011) 325-341.}

   \bibitem[IW13]{IW13}{{\sc O. Iyama} and {M. Weymss}, On the noncommutative Bondal-Orlov conjecture, \emph{J. Reine Angew. Math.} \textbf{683} (2013) 119-128.}

   \bibitem[IW14]{IW14}{{\sc O. Iyama} and {M. Weymss}, Maximal modifications and Auslander-Reiten duality for non-isolated singularities, \emph{Invent. Math.} \textbf{197} (3) (2014) 521-586.}
   
   \bibitem[IY08]{IY08}{{\sc O. Iyama} and {\sc Y. Yoshino}, Mutation in triangulated categories and rigid {C}ohen-{M}acaulay modules, \emph{Invent. Math.} \textbf{172} (2008) 117-168.}

   \bibitem[Iya07a]{Iya07a}{{\sc O. Iyama}, Higher-dimensional {A}uslander-{R}eiten theory on maximal orthogonal subcategories, \emph{Adv. Math.} \textbf{210}(1) (2007) 22-50.}

   \bibitem[Iya07b]{Iya07b}{{\sc O. Iyama}, Auslander correspondence, \emph{Adv. Math.} \textbf{210}(1) (2007) 51-82.}

   \bibitem[Iya11]{Iya11}{{\sc O.Iyama}, Cluster tilting for higher Auslander algebras, \emph{Adv. Math.} \textbf{226}(1) (2011) 1-61. }
   
   \bibitem[Jas16]{Jas16}{{\sc G. Jasso}, {$n$}-abelian and {$n$}-exact categories, \emph{Math. Z.} \textbf{283} (2016) 703-759.}
   
   \bibitem[JK16]{JK16}{{\sc G. Jasso} and {\sc J. K\"{u}lshammer}, The naive approach for constructing the derived category of a $d$-abelian category fails, arXiv:1604.03473.}
   
   \bibitem[JK19]{JK19}{{\sc G. Jasso} and {\sc J. K\"{u}lshammer}, Higher {N}akayama algebras {I}: {C}onstruction , \emph{Adv. Math.} \textbf{351} (2019) 1139-1200.}

   \bibitem[JKM22]{JKM22}{{\sc G. Jasso}, {\sc B. Keller} and {\sc F. Muro}, The triangulated Auslander-Iyama correspondence, \emph{arXiv:2208.14413}}

   \bibitem[Kva21]{Kv21}{{\sc S. Kvamme}, $d\mathbb{Z}$-cluster tilting subcategories of singularity categories, \emph{Math. Z.} \textbf{297} (2021) 803-825.}
   
   \bibitem[KS06]{KS06}{{\sc M. Kashiwara} and {\sc P. Schapira}, \emph{Categories and sheaves}, Grundlehren der mathematischen Wissenschaften [Fundamental
   Principles of Mathematical Sciences] Volumn 332, Springer-Verlag, Berlin.}

   \bibitem[Lad12]{Lad12}{{\sc S. Ladkani}, On derived equivalences of lines, rectangles and triangles, \emph{J. Lond. Math. Soc.} \textbf{87} (2012) 157-176.}

   \bibitem[OT12]{OT12}{{\sc S. Oppermann} and {\sc H. Thomas}, Higher-dimensional cluster combinatorics and representation theory, \emph{J. Eur. Math. Soc} \textbf{14} (2012) 1679-1737.}

   \bibitem[Ric89]{Ric89}{{\sc J. Rickard}, Morita theory for derived categories, \emph{J. Lond. Math. Soc.} \textbf{39} (1989) 436-456.}

   \bibitem[S\"{o}d24]{Soed24}{{\sc C. S\"{o}derberg}, Preprojective algebras of $d$-representation finite spieces with relations, \emph{J. Pure Appl. Algebra} \textbf{228} (2024) 107520.}

   \bibitem[Vas19]{Vas19}{{\sc L. Vaso}, $n$-Cluster tilting subcategories of representation-directed algebras, \emph{J. Pure Appl. Algebra} \textbf{223} (5) 2101-2122.}

   \bibitem[Xin25]{Xin25}{{\sc W.Xing}, $nd\mathbb{Z}$-cluster tilting subcategories of $d$-Nakayama algebras, In preparation.}

   

   
   


\end{thebibliography}

\providecommand{\bysame}{\leavevmode\hbox to3em{\hrulefill}\thinspace}
\renewcommand{\MRhref}[2]{%
  \href{http://www.ams.org/mathscinet-getitem?mr=#1}{#2}
}
\renewcommand\MR[1]{\relax\ifhmode\unskip\space\fi MR~\MRhref{#1}{#1}}


\end{document}